\documentclass[10pt]{amsart}
\usepackage{xypic,somedefs,tracefnt,latexsym,rawfonts,graphicx,amsfonts,amssymb,latexsym,enumerate,amscd}
\makeatletter
\@namedef{subjclassname@2020}{%
  \textup{2020} Mathematics Subject Classification}
\makeatother
\def\cal{\mathcal}
\def\Bbb{\mathbb}

\def\r{\rangle}
\def\l{\langle}

\def\vtr{\trianglelefteq}

\newtheorem{thm}{Theorem}[section]
\newtheorem{prop}[thm]{Proposition}
\newtheorem{exm}[thm]{Example}
\newtheorem{lemma}[thm]{Lemma}
\newtheorem{cor}[thm]{Corollary}
\newtheorem{defn}[thm]{Definition}
\newtheorem{rem}[thm]{Remark}
\begin{document}
\title[Configuration Lie groupoids and orbifold braid groups]{Configuration
 Lie groupoids and orbifold\\ braid groups}
\author[S.K. Roushon]{S.K. Roushon}
\address{School of Mathematics\\
Tata Institute\\
Homi Bhabha Road\\
Mumbai 400005, India}
\email{roushon@math.tifr.res.in} 
\urladdr{http://www.math.tifr.res.in/\~\ roushon/}
\date{\today}
\begin{abstract}
  We propose two definitions of configuration Lie groupoids and in
  both the cases we prove
  a Fadell-Neuwirth type fibration theorem for a class of 
  Lie groupoids. We show that this is the best possible extension,
  in the sense that, for the class of Lie
  groupoids corresponding to global quotient orbifolds with nonempty 
  singular set, the fibration theorems do not hold. 
  Secondly, we prove a short exact sequence of  
  fundamental groups (called {\it pure orbifold braid groups}) of one
  of the configuration Lie groupoids of the Lie groupoid corresponding
  to the punctured complex plane 
  with cone points. This shows the possibility of a quasifibration type Fadell-Neuwirth
  theorem for Lie groupoids.

  As consequences, first we see that the 
  pure orbifold braid groups have poly-virtually free structure, which
  generalizes the classical braid group case. We also provide an explicit set of
  generators of the pure orbifold braid groups. Secondly, we prove
  that a class of affine and finite complex Artin groups are virtually poly-free,
  which partially answers the question if all Artin groups are
  virtually poly-free ([\cite{MB}, Question 2]). Finally, combining this poly-virtually free 
  structure and a recent result (\cite{BFW}), 
  we deduce the Farrell-Jones isomorphism conjecture for the above class of orbifold braid groups.
  This also implies the conjecture for the case of the affine Artin group of type $\tilde D_n$, which
  was left open in [\cite{Rou}, Problem].\end{abstract}


\maketitle

\tableofcontents

\section{Introduction}
For a topological space $X$ and a positive integer $n$, the {\it configuration space} of
$n$-tuple of distinct points of $X$ is defined by the following.

$$PB_n(X):=X^n-\{(x_1,x_2,\ldots , x_n)\in X^n\ |\ x_i=x_j\   
\text{for some}\ i,j\in \{1,2,\ldots ,n\}\}.$$ 

\noindent
This is also called the {\it pure braid space} of $X$.

Let $M$ be a connected manifold of dimension $\geq 2$. Then, the Fadell-Neuwirth
fibration theorem ([\cite{FN}, Theorem 3]) says that the projection map $M^n\to M^{n-1}$
to the first $n-1$ coordinates defines a locally trivial fibration
$f:PB_n(M)\to PB_{n-1}(M)$, with fiber homeomorphic to $M-\{(n-1)\
\text{points}\}$.

It is an important and classical subject to study the topology and
geometry of the configuration space of $2$-dimensional manifolds.
This has applications in
different areas of Mathematics and Physics. \cite{Ar}, \cite{FN}
and \cite{Br} are among some of the 
fundamental works on this subject. The
Fadell-Neuwirth fibration theorem is the
key fundamental result, to understand the
homotopy theory of the configuration spaces of manifolds. The work in \cite{All} 
turned the attention to consider $2$-dimensional orbifolds,  
its configuration orbifold and introduced orbifold braid groups.
This has connection with some of the Artin groups (Theorem \ref{All}). On the 
other hand, a systematic and categorical way to study orbifolds is via 
Lie groupoids (\cite{IM}).

In this paper we define two notions of a 
fibration between Lie groupoids, generalizing Hurewicz fibration (Definition \ref{lf}).
We also formulate two definitions of configuration Lie
groupoids of a Lie groupoid (Definitions \ref{clg} and \ref{clgb}).
Then, we prove a Fadell-Neuwirth type fibration 
theorem (Theorem \ref{fibration}) for a class of Lie groupoids
(Definition \ref{c-group}) with respect to both
the definitions. We also show that this is the best possible class of Lie groupoids
to which the fibration theorem can be extended.  This is in the sense that the
fibration theorem fails, with respect to our definitions of a fibration,
for all Lie groupoids corresponding to global quotient 
orbifolds with nonempty singular set (Proposition \ref{mt}).

Finally, we study the fundamental
groups (called {\it pure orbifold braid group}) of the
configuration Lie groupoids (with respect to one definition) of 
Lie groupoids arising from the class $\cal S$ of genus zero $2$-dimensional
orbifolds with cone points and at least one puncture. We prove a
short exact sequence of pure orbifold braid groups (Theorem \ref{esg}).
This short exact sequence is similar to the one which 
results from the long exact sequence of homotopy groups of the
Fadell-Neuwirth fibration $f$ for the complex plane. That is, for
any member of $\cal S$, $f$ shows a
quasifibration type property in low degree. In fact, for some members
of $\cal S$, $f$ can be shown to be a quasifibration. This indicates 
the possibility of a quasifibration type Fadell-Neuwirth theorem for Lie groupoids.
See \cite{Rou21} for some more work in this direction.

By \cite{All}, elements of the pure orbifold
braid group have pictorial braid representation, which is our
main motivation for this work. We develop a simple method of
stretching a string of a braid to understand the pure orbifold braid groups. See
Lemma \ref{stretching} and Proposition \ref{gen}.

There are now multiple consequences of the above study of the pure orbifold braid groups.

\medskip
$\bullet$ On the
way of proving the above short exact sequence, we needed to find an explicit set of generators
of the pure orbifold braid group (Lemma \ref{orbi-gen}). Note that, finding an explicit set of generators
is an important problem in group theory. For the classical braid groups a set
of generators was given in \cite{Ar}. This gave us the crucial ingredient
to find the generators of the pure orbifold braid groups.

$\bullet$ From the proof of Theorem \ref{esg},
as a particular case, we provide an explicit set of
generators of the kernel of the homomorphism $f_*$ (Corollary \ref{FNF}). Giving 
an explicit set of generators of this group as a subgroup of the classical pure braid group, 
seems to be new.

$\bullet$ We see that the pure 
orbifold braid group, has a poly-virtually free structure, as well as the structure of being an iterated
semi-direct product of virtually
finitely generated free groups (Corollary \ref{vfg}). The classical pure braid 
group has such a structure but with finitely generated free
groups (\cite{CS}). 

$\bullet$ The poly-virtually free structure of the pure orbifold braid group is also used
in Theorem \ref{mtpf}, to prove that a class of affine and finite complex Artin groups are
virtually poly-free. This gives a partial answer to
the question, if all Artin groups are virtually poly-free ([\cite{MB}, Question 2]).
A recent work
in this direction is \cite{BMP}, also see \cite{W}
for a simple proof of the result of \cite{BMP}.
Poly-free groups have nice properties like locally indicable and right orderable. Also, 
an inductive argument using [\cite{DS}, Theorem 2.3] shows that 
a virtually poly-free group has finite asymptotic dimension.

$\bullet$ Our final application of the poly-virtually free structure of the pure
orbifold braid group, is to prove the Farrell-Jones isomorphism conjecture for any group which
contains the pure orbifold braid group (Theorem \ref{FJC}),
as a subgroup of finite index.
Consequently, this settles the case of the Artin group
of type $\tilde D_n$, left open in [\cite{Rou}, Problem]. There is much interest in proving the Farrell-Jones
conjecture for Artin groups in recent times. Also see 
\cite{HO} and \cite{W}.

We conclude the introduction with a few words on the organization of
the paper. In the next section
we state our main results, and also recall some background required for the
statements. Section 3 contains some more background materials on Lie
groupoids and proofs of some lemmas left in Section 2. In this section we
also formalize the results we have proved, in the category of orbifold groupoids.
In Section 4, we give a detailed introduction
to orbifold braid groups and use the stretching technique to prove some crucial results.
The proofs of the main results are given in Section 5.

\section{Statements of results} \label{sor}
We start with some conventions we will follow. All manifolds
are assumed to be smooth, Hausdorff and
paracompact, and by a `map' we will always mean a smooth map.
Also by a `fibration' we will mean a Hurewicz fibration. 

Let $\cal G$ be a Lie groupoid with object space ${\cal G}_0$, and 
morphism space ${\cal G}_1=\cup_{x,y\in {\cal G}_0}mor_{\cal G}(x,y)$. 
$s, t:{\cal G}_1\to {\cal G}_0$ are 
the source and the target maps, that is for a morphism $\alpha\in mor_{\cal G}(x,y)$,
$s(\alpha)=x$ and $t(\alpha)=y$. $s$ and $t$ are assumed to be smooth
submersions. A {\it homomorphism} $f$ between two Lie
groupoids is a smooth functor, and we denote the object level map by
$f_0$, and the morphism level map by $f_1$. For any $x\in {\cal G}_0$, 
the set $t(s^{-1}(x))$ of all points which are the targets 
of morphisms emanating from $x$, is called the {\it orbit} of $x$. The 
space (with quotient topology) of all orbits is called the 
{\it base space}, and is denoted by $|{\cal G}|$.

\begin{defn}\label{c-group} {\rm A Lie groupoid $\cal G$ is called a {\it c-groupoid}, if 
the above quotient map $\kappa:{\cal G}_0\to |{\cal G}|$ is a covering map.
$\cal G$ is said to be {\it Hausdorff} if the space
$|{\cal G}|$ is Hausdorff. Clearly, a $c$-groupoid is Hausdorff.}\end{defn}

For $x\in {\cal G}_0$, the {\it star} $St_x$ at $x$ is defined by
$s^{-1}(x)=\cup_{y\in {\cal G}_0}mor_{\cal G}(x,y)$. $St_x$ is a
submanifold of ${\cal G}_1$, since $s$ is a submersion. The dimension of
the manifold ${\cal G}_0$ is called the {\it dimension} of the Lie groupoid $\cal G$.
The {\it isotropy} group at $x\in {\cal G}_0$ is defined by ${\cal G}_x:=mor_{\cal G}(x,x)$.
${\cal G}_x$ is a Lie group.

Our main examples of Lie groupoids for this work is the following.

\begin{exm}\label{tg}{\rm Let $M$ be a manifold and a Lie group $H$ 
    acting on $M$ smoothly. We construct a Lie groupoid ${\cal G}(M,H)$, called
    the {\it translation groupoid}, 
    out of this information. Define ${\cal G}(M,H)_0=M$, ${\cal G}(M,H)_1=H\times M$,
    $s(h,x)=x$, $t(h,x)=h(x)$, $u(x)=(1,x)$, $i(h,x)=(h^{-1},x)$ and $(h',h(x))\circ (h,x)=(h'h,x)$.
  When $H$ is the trivial group then ${\cal G}(M, H)$ is called the {\it
    unit groupoid}, and is identified with $M$. If $H$ is discrete
  and acts on $M$ freely and properly
  discontinuously, then ${\cal G}(M,H)$ is a $c$-groupoid. If the action is effective and
  properly discontinuous then, $M/H$ is an orbifold and hence
  ${\cal G}(M,H)$ is Hausdorff ([\cite{Thu}, Proposition 5.2.6]).}\end{exm}
    
\subsection{Fibrations of Lie groupoids and the Fadell-Neuwirth fibration theorem}
Defining a covering map between orbifolds is well-known, it satisfies properties parallel
to an ordinary covering map of spaces (see [\cite{Thu}, Definition 5.3.1]). 
In [\cite{Thu}, Definition 5.11.1] a definition of an orbifold fiber bundle was given but with
fiber a manifold. There are other definitions of an orbifold fibration in
the literature, but those are with 
assumptions suitable for some specific purposes.
For example, these definitions either give only `submersion' or
`path lifting property' when restricted
to manifolds.

To define a morphism between orbifolds, independent of atlases,
we need to first look at orbifolds as Lie groupoids. This problem was
solved in \cite{IM}, by considering the category of Lie groupoids and generalized
maps. A {\it generalized map} is a certain equivalence class of
homomorphisms. In the next section, we will describe the concept of `orbifold
groupoids' which helps in accommodating orbifolds in the category of
Lie groupoids.

We now proceed to give two definitions of a fibration for orbifolds, in the set up 
of Lie groupoids, which restrict to Hurewicz fibration
when there is no singularity in the orbifolds, that is, for
manifolds.

To define one of the notions of a fibration between Lie groupoids, we
need to generalize the construction of a
Lie groupoid from 
a Lie group action on a manifold as in Example \ref{tg}, to a Lie
groupoid action on a manifold.

\begin{defn}\label{leftspace}{\rm 
Given a Lie groupoid $\cal G$ and a manifold $M$, $M$ is called a {\it (left) $\cal G$-space} if
there is a smooth map $\pi:M\to {\cal G}_0$, and
an {\it action map}  
$$\mu:{\cal G}_1\times_{{\cal G}_0}M\to M$$ satisfying the following
properties. For $\alpha\in {\cal G}_1$ and $x\in M$ with $\pi(x)=s(\alpha)$ (defining the fibered product),
$\pi(\mu(\alpha,x))=t(\alpha)$. Finally, if $\pi(x)=y$ then $\mu(1_y,x)=x$ and 
$\mu(\alpha,\mu(\beta,x))=\mu(\alpha\beta,x)$ for $\beta\in mor_{\cal G}(y,w)$ and $\alpha\in
mor_{\cal G}(w,z)$.}\end{defn}

\begin{defn} \label{lf}{\rm Let $\cal H$ and $\cal G$ be two Lie
    groupoids and $f:{\cal H}\to {\cal G}$ a homomorphism, so that $f_0:{\cal
      H}_0\to {\cal G}_0$ is a submersion fibration. 

$\bullet$ $f$ is called an {\it
      a-fibration} if  ${\cal H}_0$ is a
    $\cal G$-space with $f_0=\pi$, ${\cal H}_1={\cal G}_1\times_{{\cal
        G}_0}{\cal H}_0$ and $f_1$ is the first projection. The source and the target
  maps of $\cal H$ are respectively the second projection and the
  action map.

$\bullet$ $f$ is called a {\it b-fibration} if the map 
${\cal H}_1\to {\cal G}_1\times_{{\cal G}_0}{\cal H}_0$ defined by 
$\alpha\mapsto (f_1(\alpha),s(\alpha))$ is a surjective submersion.}\end{defn}

\begin{rem}{\rm The idea of an $a$-fibration comes from the way
    covering Lie groupoid of a 
    Lie groupoid is defined. See [\cite{IM}, \S5.3]. The same framework works
    to define a fiber bundle or a 
    vector bundle on a Lie groupoid also. And the $b$-fibration
    definition is primarily motivated from \cite{HF}. It can be shown that the second condition in the
    definition of a $b$-fibration is equivalent
    to demanding that $f_1|_{St_x}:St_x\to St_{f_0(x)}$ is a
    submersion for all $x\in {\cal H}_0$.}\end{rem}

Now, we are in a position to define the configuration Lie groupoids 
of a Lie groupoid.

The first definition is given below, which
facilitates to prove an $a$-fibration type Fadell-Neuwirth fibration theorem.

\begin{defn}\label{clg}{\rm Let $\cal G$ be a Hausdorff Lie groupoid. 
The {\it a-configuration Lie groupoid} $PB^a_n({\cal G})$ of 
    $\cal G$ has the object space defined as follows.

  $$PB_n({\cal G})_0:={\cal G}_0^n-\{(x_1,x_2,\ldots , x_n)\in {\cal
    G}_0^n\ |\ t(s^{-1}(x_i))=t(s^{-1}(x_j))\ $$$$ 
\text{for some}\ i,j\in \{1,2,\ldots ,n\}\}.$$  
Since $\cal G$ is Hausdorff, $PB_n({\cal G})_0$ is an open set 
in ${\cal G}_0^n$. This follows from the fact that  
$(\kappa^n)^{-1}(PB_n(|{\cal G}|))=PB_n({\cal G})_0$ and $PB_n(|{\cal G}|)$
is open in $|{\cal G}|^n$, as $|{\cal G}|$ is Hausdorff.
Therefore, $PB_n({\cal G})_0$ is a smooth manifold.

The morphism space, the source and the target maps are defined inductively. We denote these 
 maps for $PB^a_n({\cal G})$ with a subscript `$n$', with the understanding 
that the corresponding maps with subscript `$1$' are those of $\cal G$.

For $n=1$, by convention $PB_1({\cal G})_0={\cal G}_0$. Therefore, define
$PB^a_1({\cal G})_1={\cal G}_1$. Hence, $PB^a_1({\cal G})$ is a Lie groupoid.

Assume, we have defined $PB^a_{n-1}({\cal G})_1$ and the structure maps so that
$PB^a_{n-1}({\cal G})$ is a Lie groupoid.

Define
$$PB^a_n({\cal G})_1:=PB^a_{n-1}({\cal G})_1\times_{PB_{n-1}({\cal G})_0}PB_n({\cal G})_0,$$
which is the fibered product of the source map
$s_{n-1}:PB^a_{n-1}({\cal G})_1\to PB_{n-1}({\cal G})_0$ and
the projection $PB_n({\cal G})_0\to PB_{n-1}({\cal G})_0$ to the first $n-1$ coordinates.
Note that, $PB^a_n({\cal G})_1$ is a smooth manifold because the projection map and the source map are
submersions.

Let $\alpha_n=(\alpha_{n-1},x)\in PB^a_n({\cal G})_1$, then $t_n$ and $s_n$ are defined as follows.
$$t_n(\alpha_n)=(t_{n-1}(\alpha_{n-1}), x_n), s_n(\alpha_{n-1},x)=x,$$
here $x=(x_1,x_2,\ldots, x_n)\in PB_n({\cal G})_0$. Clearly,
$t_n$ and $s_n$ are both smooth, but 
we still have to check that $t_n(\alpha_n)$ lies in the appropriate space. That is, $x_n$ and the
coordinates of $t_{n-1}(\alpha_{n-1})$ have distinct orbits. Note that, by induction,
$$\alpha_{n-1}=(\alpha_1,(x_1,x_2),(x_1,x_2,x_3),\cdots, (x_1,x_2,\ldots ,x_{n-1})),$$
for some $\alpha_1\in{\cal G}_1$ with $s_1(\alpha_1)=x_1$.
Therefore,
$t_n(\alpha_n)=(t_1(\alpha_1),x_2,\ldots, x_{n-1}, x_n)$.
We have to check $t_1(\alpha_1)$ and $x_n$ have distinct orbits, which is clear, since $x_1$ and
$t_1(\alpha_1)$ have the same orbit and $x_1$ and $x_n$ have distinct orbits. This completes
the definition.}\end{defn}

The following lemma gives the homomorphism which we need to show to be an $a$-fibration.

\begin{lemma} \label{pb} The projection $PB_n({\cal G})_0\to PB_{n-1}({\cal G})_0$
  to the first $n-1$ coordinates and
  the projection $PB^a_n({\cal G})_1\to PB^a_{n-1}({\cal G})_1$ to the
  first coordinate define a homomorphism
  $$F^a:PB^a_n({\cal G})\to PB^a_{n-1}({\cal G})$$ of Lie groupoids.\end{lemma}

Now, we give the second definition of the configuration Lie groupoid, which will
give a $b$-fibration type Fadell-Neuwirth fibration theorem.

\begin{defn}\label{clgb}{\rm We define the {\it b-configuration Lie groupoid}
    $PB_n^b({\cal G})$ of a Hausdorff Lie groupoid $\cal G$ as follows. The object space is
    the same as that of an a-configuration Lie groupoid, that is $PB_n({\cal G})_0$.
    Define the morphism space by 
$$PB_n^b({\cal G})_1:=(s^n, t^n)^{-1}(PB_n({\cal G})_0\times PB_n({\cal G})_0).$$
Here 
$(s^n, t^n):{\cal G}^n_1\to {\cal G}^n_0\times {\cal G}^n_0$
is defined coordinate-wise, that is $$(s^n,t^n)(\alpha_1,\alpha_2,\ldots, \alpha_n)=
((s(\alpha_1),s(\alpha_2),\ldots, s(\alpha_n)),(t(\alpha_1),t(\alpha_2),\ldots, t(\alpha_n))).$$}\end{defn}

Next, we have to give the homomorphism in this context, which we need to show to be a $b$-fibration.

\begin{lemma} \label{pb1} The projection maps 
  $PB_n({\cal G})_0\to PB_{n-1}({\cal G})_0$ and $PB^b_n({\cal G})_1\to PB^b_{n-1}({\cal G})_1$
  both to the first $n-1$ coordinates, define a homomorphism
  $$F^b:PB^b_n({\cal G})\to PB^b_{n-1}({\cal G})$$ of Lie groupoids.\end{lemma}

Proofs of Lemmas \ref{pb} and \ref{pb1} will be given in the next
section.

Now, we are in a position to state our first theorem.
Recall the definition of a $c$-groupoid from Definition \ref{c-group}.

\begin{thm} \label{fibration} Let $\cal G$ be a 
  c-groupoid of dimension $\geq 2$. Then, the homomorphism
  $F^*:PB^*_n({\cal G})\to PB^*_{n-1}({\cal G})$ is a $*$-fibration.
  Here, $F^*$ is defined in Lemmas \ref{pb} and \ref{pb1}, for $*=a,b$, respectively.\end{thm}

\subsection{Counter examples and a short exact sequence}\label{counter}
Recall that an orbifold is called {\it good} if its universal cover is a manifold.
An orbifold $S$ is called a {\it global quotient} if
there is a finite group $H$ acting effectively
on a connected manifold $M$, such that $S=M/H$. That is, $M\to S$ is an
orbifold covering map. For example, any $2$-dimensional good orbifold with finitely generated
orbifold fundamental group is a global quotient ([\cite{Sco83}, p. 426]). Note that, an orbifold
is Hausdorff, and hence if $S$ is a global quotient as above, we can
construct the configuration Lie groupoids of the Lie groupoid ${\cal G}(M,H)$.
See [\cite{Thu}, \S5.2] for more on orbifolds.

The following Proposition shows that the Fadell-Neuwirth fibration
theorem is not extendable to all Hausdorff Lie groupoids. The $c$-groupoids
are the best possible cases after the unit Lie groupoids.
The main obstruction
in these examples is that the quotient map ${\cal G}(M,H)_0\to |{\cal G}(M,H)|$ is not a
genuine covering map.

\begin{prop}\label{mt} Let $S$ be a global quotient orbifold of dimension $\geq 2$, with
  at least one singular point. Let $M$ and $H$ be as defined above. Then, the homomorphism  
    $$F^*:PB^*_n({\cal G}(M, H))\to PB^*_{n-1}({\cal G}(M, H)),$$ is not a $*$-fibration
    of Lie groupoids. In fact, the object level map
    $F^*_0$ is not even a quasifibration. Here, for $*=a,b$, $F^*$
  is defined in Lemmas \ref{pb} and \ref{pb1}, respectively.\end{prop}

We now recall the definition of the homotopy groups of a Lie groupoid.

Given a Lie groupoid $\cal G$, consider the following $n$-times iterated fibered product manifold 
  $${\cal G}_n={\cal G}_1\times_{{\cal G}_0}{\cal G}_1\times_{{\cal
      G}_0}\cdots \times_{{\cal G}_0} {\cal G}_1.$$
  Then ${\cal G}_{\bullet}$ defines a simplicial manifold and its geometric realization is defined as
  the {\it classifying space} $B{\cal G}$ of $\cal G$. Up to weak homotopy equivalence
  this space is unique in the equivalence class (see [\cite{IM},
  \S2.4, \S4.3]) 
  of $\cal G$. The homotopy groups of
  $\cal G$ are then defined as follows. See [\cite{IM}, \S4.3]. $$\pi_k({\cal G}, *):=\pi_k(B{\cal G}, *).$$
  An $a$ or a $b$-fibration between Lie groupoids induces a quasifibration on their
  classifying spaces ([\cite{May72}, Theorem 12.7]). Recall that, a  
{\it quasifibration} is a map which induces a long exact sequence
of homotopy groups, similar to the one induced by a (Serre) fibration, equivalently,
for path connected base, if the homotopy fiber is weak homotopy equivalent to a fiber.
Hence, for c-groupoids, we can calculate the homotopy
groups of the $a$ or the $b$-configuration Lie groupoids inductively using the
long exact sequence of homotopy groups. 

\begin{rem}\label{remm}{\rm Let $f:{\cal H}\to {\cal G}$ be an $a$ or a $b$-fibration
    between Lie groupoids. The {\it fiber}
    over $x\in {\cal G}_0$ is the Lie groupoid $\cal F$ whose object space
    is the manifold ${\cal F}_0=f_0^{-1}(x)$ (since $f_0$ is a submersion), and
    the morphism space is ${\cal F}_1=f_1^{-1}({\cal G}_x)$.
    Since $f_1$ is a submersion, ${\cal F}_1$ is a manifold.

    The long exact sequence  of homotopy groups of an $a$ or a $b$-fibration, 
mentioned above, then 
involves the homotopy groups of this fiber Lie groupoid.}\end{rem}

Next, we give an infinite series of examples of
$2$-dimensional Hausdorff Lie groupoids $\cal G$ which
are not c-groupoids, but still we can deduce a short exact
sequence connecting the fundamental groups 
of $PB^b_n({\cal G})$, $PB^b_{n-1}({\cal G})$, and the fundamental group of
the fiber of the homomorphism $F^b:PB^b_n({\cal G})\to PB^b_{n-1}({\cal G})$.

\begin{exm}\label{mainexm}{\rm We denote by ${\Bbb C}(k,m;q)$ the complex plane with $k$ punctures at
the points $p_1,p_2,\ldots, p_k\in {\Bbb C}$, $m$ marked points (called {\it cone} points)
$x_1,x_2,\ldots, x_m\in {\Bbb C}-\{p_1,p_2,\ldots, p_k\}$, and an integer $q_i>1$ attached
to $x_i$ for $i=1,2,\ldots, m$. $q_i$ is called the {\it order} of the cone point $x_i$. Here
$q$ denote the $m$-tuple $(q_1,q_2,\ldots , q_m)$.
Note that, ${\Bbb C}(k,m;q)$ is a $2$-dimensional orbifold. In the Introduction we
denoted this class of orbifolds by $\cal S$.
We now show that ${\Bbb C}(k,m;q)$ is a good orbifold. Let $B$ be a
    big enough closed disk in $\Bbb C$, which contains the
    punctures and cone points of ${\Bbb C}(k,m;q)$ in its
    interior. Remove small disjoint open disks (contained
    in the interior of $B-\{x_1,x_2,\ldots, x_m\}$)
    around the puncture points, name the resulting space $B'$. Let
    $DB'$ 
be the double of $B'$. Then $DB'$ is a closed and good $2$-dimensional
orbifold. This follows from [\cite{Thu}, p. 237]. Hence $DB'$ has a finite
sheeted orbifold covering which is a manifold. Therefore, we
conclude that ${\Bbb C}(k,m;q)$ also has a finite sheeted orbifold
covering which is a manifold. Hence, there is a $2$-manifold
$S(k,m;q)$ and a finite group 
$H(k,m;q)$ acting effectively on $S(k,m;q)$, with quotient
${\Bbb C}(k,m;q)$. Hence, by Example
\ref{tg}, ${\Bbb C}(k,m;q)$ is realized as the Lie groupoid 
${\cal G}(S(k,m;q),H(k,m;q))$. Note that, the Lie groupoid
${\cal G}(S(k,m;q), H(k,m;q))$ is Hausdorff. Therefore, we can consider the
configuration Lie groupoids of this Lie groupoid.}\end{exm}

Now, we recall some consequences of the Fadell-Neuwirth fibration theorem
for $2$-manifolds, to motivate our next theorem.

Let $N$ be a connected $2$-manifold and $M=N-\{s-\text{points}\ , s\geq 1\}$.
Then, using the long exact sequence of
homotopy groups for the Fadell-Neuwirth fibration for $M$, we get the following 
split short exact sequence ([\cite{FN}, Theorem 3]).

\centerline{
\xymatrix{1\ar[r]&\pi_1(F)\ar[r]&\pi_1(PB_n(M))\ar[r]^{f_*}&\pi_1(PB_{n-1}(M))\ar[r]&1.}}
\noindent
Here $F=M-\{(n-1)-\text{points}\}$. The second surjective homomorphism is induced
by the fibration map $f:PB_n(M)\to PB_{n-1}(M)$. We will give an explicit
pictorial description of
this homomorphism in terms of braids, when $M={\Bbb C}$.

The above exact sequence also gives an interesting, and useful poly-free as well as
`iterated semi-direct product of finitely
generated free groups structure' on the pure braid group $\pi_1(PB_n(M))$. See \cite{CS} for more
on this subject.

Proposition \ref{mt} says that the Lie groupoids  
${\cal G}(S(k,m;q), H(k,m;q))$, for $m\geq 1$, do not
have a Fadell-Neuwirth type fibration. Nevertheless, we
can still prove a short exact sequence similar to the one above,
in the following theorem, directly. This theorem shows that $F^b$ is a
kind of `quasifibration' in low degree.

  \begin{thm}\label{esg} There is a split exact sequence of fundamental groups of
    Lie groupoids as follows.
    
\centerline{
\xymatrix{1\ar[r]&K\ar[r]&\pi_1(PB^b_n({\cal
    X}))\ar[r]^{F^b_*}&\pi_1(PB^b_{n-1}({\cal X}))\ar[r]&1.}}
\noindent
Here, $K$ is isomorphic to $\pi_1({\cal F})$, 
$${\cal X}={\cal G}(S(k,m;q), H(k,m;q))$$ and $${\cal F}={\cal G}(S(k+n-1,m;q), H(k+n-1,m;q)).$$
\end{thm}

\begin{rem}\label{es} {\rm The orbifold fundamental group
    of an orbifold, as defined in [\cite{Thu}, Definition 5.3.5],
    is identified with the fundamental group of an 
    associated orbifold groupoid (see Example \ref{ohg}). We just need
    to note here
    that, ${\cal X}$ and $PB_n^b({\cal X})$ are examples of orbifold
    groupoids. Furthermore, the 
    $b$-configuration Lie groupoid of ${\cal X}$ is the correct model
    of a Lie groupoid inducing the orbifold structure on $PB_n({\Bbb C}(k,m; q))$.
    Hence, $\pi_1(PB^b_n({\cal X}))$ is isomorphic to
    $\pi_1^{orb}(PB_n({\Bbb C}(k,m; q)))$, since $PB_n^b({\cal G}(M,H))={\cal
    G}(PB_n(M), H^n)$. See Example \ref{tg} for notation. 

Hence, the above exact sequence in
  Theorem \ref{esg} is
  equivalent to the following.
  
\centerline{
\xymatrix{1\ar[r]&K\ar[r]&\pi_1^{orb}(PB_n(S))\ar[r]&\pi_1^{orb}(PB_{n-1}(S))\ar[r]&1.}}
\noindent
Here, $K$ is isomorphic to $\pi_1^{orb}(F)$, $S={\Bbb C}(k,m; q)$ and
$F=S-\{(n-1)-\text{regular points}\}$. By {\it regular points}
we mean points which are not singular points in an orbifold, that is, 
in this case these are points in ${\Bbb C}-\{x_1,\ldots,
x_m,p_1,\ldots, p_k\}$.
That is, $F={\Bbb C}(k+n-1,m;q)$. In fact, we
will give the proof of this exact sequence. This exact
sequence can also be obtained for the genus zero $2$-dimensional
orbifold $S$ with countably infinite number of punctures and cone points.
This can be done by writing $S$ as an infinite increasing union of
orbifolds of the form ${\Bbb C}(k,m;q)$, and then taking a direct limit,
since direct limit of a directed system of exact sequences is exact again.}\end{rem}

Consider the free action of the symmetric group $S_n$ on $PB_n({\Bbb C}(k,m; q))$ by permuting
the coordinates. The quotient orbifold is denoted by $B_n({\Bbb C}(k,m; q))=PB_n({\Bbb C}(k,m; q))/S_n$.

\begin{defn}\label{pobg}{\rm  
The group 
$\pi_1^{orb}(PB_n({\Bbb C}(k,m; q)))$ is called the {\it pure orbifold braid group} of the
orbifold ${\Bbb C}(k,m;q)$, and $\pi_1^{orb}(B_n({\Bbb C}(k,m; q)))$ is called
its {\it orbifold braid group}.}\end{defn}

We get the following consequence on the structure
of the pure orbifold braid group, from
Remark \ref{es}.

Before that we recall the following definition.

\begin{defn}\label{VPF} {\rm Let $\cal F$ and $\cal {VF}$ denote the class of free groups
    and virtually-free groups, respectively. Let $\cal C$ be either $\cal F$ or $\cal {VF}$.
    A group $G$ is
    called {\it poly}-$\cal C$, if $G$ 
    admits a normal series $1=G_0\vtr G_1\vtr G_2\vtr \cdots \vtr G_n=G$, such that 
    $G_{i+1}/G_i\in {\cal C}$, for $i=0,1,\ldots, n-1$. The minimum such $n$ is called the
  {\it length} of the poly-$\cal C$ structure. $G$ is called
    {\it virtually poly}-$\cal C$ if $G$ contains a finite index poly-$\cal C$
    subgroup.}\end{defn}

Note that, a subgroup of a poly-$\cal C$ group is poly-$\cal C$ and  
an extension of a poly-$\cal C$ group by a poly-$\cal C$ group is poly-$\cal C$. Also, 
the class of virtually-$\cal F$ groups is closed under taking finite free products.

\begin{cor}\label{vfg} The group 
$\pi_1^{orb}(PB_n({\Bbb C}(k,m; q)))$ has a poly-$\cal {VF}$ structure, consisting of finitely
  presented subgroups in a normal series. Furthermore, it has  
an iterated semi-direct product of virtually finitely generated free group structure.
\end{cor}

\begin{proof} Note that, for all $k,m$ and $q$, $\pi_1^{orb}({\Bbb C}(k,m;q))$
  is isomorphic to the free
  product of the finite cyclic groups of order $q_i$, for $i=1,2,\ldots,
  m$, and a free group on $k$ generators. Hence the Corollary follows
  from Remark \ref{es} and by induction on $n$.\end{proof}

\subsection{Virtual poly-freeness of affine Artin groups}\label{aag}
We first recall some basics related to Artin groups. For details on this
subject see \cite{Hum}, \cite{Bri1} and \cite{Br}.

Let $K=\{s_1,s_2,\ldots,s_k\}$ be a finite set, and $m:K\times K\to
\{1,2,\cdots, \infty\}$ be a 
map with the property that $m(s,s)=1$, and $m(s',s)=m(s,s')\geq 2$ for $s\neq s'$. 
The {\it Coxeter group} associated to the pair $(K,m)$ is by definition the 
following group.
$${\cal W}_{(K,m)}=\langle K\ |\ (ss')^{m(s,s')}=1,\ s,s'\in S\ \text{and}\ m(s,s')<\infty\rangle.$$
A complete 
classification of finite, irreducible Coxeter groups is known (\cite{Cox}). A Coxeter
group is called {\it irreducible} if it is not the direct product of two non-trivial
Coxeter groups. Finite  Coxeter groups are exactly the finite reflection groups.
Also, there are infinite Coxeter groups which are affine reflection groups (\cite{Hum}). 

The {\it Artin group} associated to the Coxeter
group ${\cal W}_{(K, m)}$ is, by definition, 
$${\cal A}_{(K, m)}=\l K\ |\ ss'ss'\cdots = s'ss's\cdots;\ s,s'\in K\r.$$
Here, the number of times the factors in $ss'ss'\cdots$ appear is 
$m(s,s')$; e.g., if $m(s,s')=3$, then the relation is $ss's=s'ss'$. 
${\cal A}_{(K,m)}$ is called the Artin group of type 
${\cal W}_{(K,  m)}$.

A {\it finite type} or an {\it affine type} Artin group is by definition
the Artin group corresponding to a finite or affine type Coxeter group, respectively.
There are {\it complex type} Artin groups also, which are the Artin
groups whose corresponding Coxeter group is generated by reflections
along complex hyperplanes in some complex space. 

It is still an open question if all Artin groups are virtually poly-$\cal F$. See
[\cite{MB}, Question 2].
Among the finite type Artin groups, 
the groups of types $A_n$, $B_n (=C_n)$, $D_n$,
$F_4$, $G_2$ and $I_2(p)$ are already known to be virtually poly-$\cal F$ (\cite{Br}).

Here, we extend this class and prove the following theorem.

\begin{thm}\label{mtpf} Let $\cal A$ be an Artin group of the affine type
  $\tilde A_n$, $\tilde B_n$, $\tilde C_n$, $\tilde D_n$ or of the finite
  complex type $G(de,e,r)$ ($d,r\geq 2$).
  Then, $\cal A$ is virtually poly-$\cal F$.\end{thm} 

Recently, it was shown  
in \cite{BMP} that the even Artin groups (that is when $m(s,s')=2$ for all $s\neq s'$)
of $FC$-types (certain amalgamation of finite type even Artin groups) are poly-$\cal F$.
A simple proof of this result of \cite{BMP} is given in \cite{W}.

\subsection{The Farrell-Jones isomorphism conjecture}

Using Corollary \ref{vfg} and a recent result (\cite{BFW}), we prove
the following theorem. Before we give the statement of the result, 
we recall that the Farrell-Jones
isomorphism conjecture is an important conjecture in Geometry and
Topology, and much works have been done in recent times. The conjecture
implies some of the classical conjectures in Topology, like Borel and Novikov
conjectures, and provides a better understanding of the $K$- and
$L$-theory of a group. 

\begin{thm}\label{FJC} The
  Farrell-Jones isomorphism conjecture with coefficients and finite wreath product,
  is true for the orbifold braid group of the
orbifold ${\Bbb C}(k,m; q)$. Consequently, it is true  
for the Artin group of type $\tilde D_n$.\end{thm}

We recall that the case of the Artin group of type $\tilde D_n$ was left open
in \cite{Rou}. See Problem at the end of \cite{Rou}. 

The proof of Theorem \ref{FJC} is short and does not require the exact statement of the 
conjecture, but needs some widely published hereditary properties of
the conjecture,  
and some well-known results in this area. Therefore, we do not state  
the conjecture, and refer the reader to \cite{L} or \cite{Rou1} for more on this subject. 

\section{Some basics}\label{of}
In this section we recall some more basics on Lie groupoids and complete
some checking, left in the last section, including the proofs of Lemmas \ref{pb} and \ref{pb1}.

\subsection{Lie groupoids} For the material recalled in this subsection see
\cite{ALR07} or \cite{IM}.

A {\it groupoid} is a small category $\cal G$ with all morphisms invertible. We denote
the object set by ${\cal G}_0$ and the union of all morphism sets
by ${\cal G}_1:=\cup_{x,y\in {\cal G}_0} mor_{\cal G}(x,y)$. There are
the following structure maps to define a groupoid.

$(ST).$ $s,t:{\cal G}_1\to {\cal G}_0$ are
defined by $s(\alpha)=x, t(\alpha)=y$,  if $\alpha\in mor_{\cal G}(x,y)$.
$\alpha$ is written as $\alpha:x\to y$.
These are respectively called the {\it source} and {\it target} maps.

$(I).$ $i:{\cal G}_1\to {\cal G}_1$ defined by $i(\alpha):=\alpha^{-1}\in mor_{\cal G}(y,x)$  if
$\alpha\in mor_{\cal G}(x,y)$. $i$ is called the {\it inverse} map.

$(M).$ $m:{\cal G}_1\times_{{\cal G}_0} {\cal G}_1\to {\cal G}_1$ is denoted
by $m(\alpha,\beta):=\beta\circ \alpha\in mor_{\cal G}(x,z)$
if $\alpha\in mor_{\cal G}(x,y)$ and $\beta\in mor_{\cal G}(y,z)$. This is called
the {\it multiplication} or {\it composition} map. Here, 
${\cal G}_1\times_{{\cal G}_0} {\cal G}_1=\{(\alpha,\beta)\in {\cal G}_1\times {\cal G}_1\ |\ t(\alpha)=s(\beta)\}$.

$(U).$ $u:{\cal G}_0\to {\cal G}_1$ defined by $u(x)=id_x\in mor_{\cal G}(x,x)$, called the {\it unit} map.

These maps should satisfy the following.

$(C).$ The multiplication is associative, that is, $f\circ(g\circ h)=(f\circ g)\circ h$
whenever they are defined. The unit map is a two-sided unit of the composition, which means 
for all $x,y\in {\cal G}_0$ and $\alpha:x\to y$, $s(u(x))=x=t(u(x))$ and $\alpha\circ (u(x))=\alpha=u(y)\circ \alpha$.
Finally, $\alpha^{-1}$ is a two-sided inverse of $\alpha$. That is $\alpha\circ \alpha^{-1}=u(y)$ and $\alpha^{-1}\circ \alpha=u(x)$.

\begin{defn}{\rm A groupoid 
$\cal G$ is called a {\it Lie groupoid} if ${\cal G}_0$ and ${\cal G}_1$ are smooth manifolds, 
all the structure maps are smooth, and in addition $s$ and $t$ are submersions. The last condition
is necessary to make sure that the fiber product
${\cal G}_1\times_{{\cal G}_0} {\cal G}_1$ is a smooth manifold.}\end{defn}

Now, recall that if a discrete group $H$ acts on a manifold $M$ effectively and properly
discontinuously, then the quotient $M/H$ has an orbifold structure.
In Example \ref{tg} we have seen how to associate a Lie groupoid
to this data, in a more general setting.

\begin{defn}\label{orbgroup}{\rm ([\cite{IM}, \S1.5])
A Lie groupoid is called {\it proper} if the map $(s,t):{\cal G}_1\to
{\cal G}_0\times {\cal G}_0$ is a proper map.  Consequently, for a proper
Lie groupoid ${\cal G}_x$
is a compact Lie group, for all $x\in {\cal G}_0$. $\cal G$ is
called a {\it foliation groupoid}, if ${\cal G}_x$ is
discrete for all $x\in {\cal G}_0$. An {\it orbifold
  groupoid} is by definition a proper foliation groupoid. Hence, 
an orbifold groupoid has finite isotropy groups. For an orbifold groupoid $\cal G$,
$|{\cal G}|$ is the orbifold, on which $\cal G$ is the {\it Lie
  groupoid structure}.}\end{defn}

\begin{rem}{\rm An orbifold groupoid is Hausdorff since it is proper.}\end{rem}

\begin{exm}\label{ohg}{\rm If a discrete group $H$ acts on a manifold $M$ effectively and 
    properly discontinuously, then ${\cal G}(M,H)$ is an orbifold
    groupoid, and gives an orbifold structure on $M/H$.
    Conversely, given an orbifold, there is a
translation Lie groupoid, namely the frame bundle manifold with the
orthogonal group action, which induces the orbifold structure on
the orbifold. Note that, any two orbifold groupoids induce 
equivalent orbifold structure on a space if and only if they are Morita equivalent.
Also Morita equivalent orbifold groupoids have isomorphic homotopy groups. Therefore,
given an orbifold groupoid one can define the {\it orbifold homotopy groups} of the
associated orbifold, as the homotopy groups of the classifying space of the orbifold
groupoid. See \cite{IM}.}\end{exm}

In the language of orbifold groupoids, we can now use the above remarks to 
formalize the main results we have proved.

\begin{thm}\label{formal}  Let $\cal G$ be a connected 
  Hausdorff Lie groupoid.

  (Theorem \ref{fibration}) Assume that
  ${\cal G}_0\to |{\cal G}|$ is a covering map, that is, the orbifold
  $|{\cal G}|$ does not have any singular point, consequently, it is a
  manifold. Then $F^a$($F^b$) is an $a$($b$)-fibration.

  Now let $\cal G$ be an orbifold groupoid.
  
  (Theorem \ref{esg}) Assume that $|{\cal G}|$ is equivalent
  to ${\Bbb C}(k,m;q)$ as an orbifold, then
  there is an exact sequence of fundamental groups as follows.

  \centerline{
    \xymatrix{1\ar[r]&\pi_1({\cal F})\ar[r]&\pi_1(PB_n^b({\cal G}))\ar[r]^{F^b_*}&\pi_1(PB_{n-1}^b({\cal G}))\ar[r]&1.}}

  \noindent
  Here $\cal F$ is a fiber of the homomorphism $F^b$.

  (Proposition \ref{mt}) Assume that 
  ${\cal G}_0\to |{\cal G}|$ is an orbifold covering map,
  $|{\cal G}|$ is a global quotient and has a nonempty singular
  set. Then, we can find another
  orbifold groupoid $\cal H$,  
  Morita equivalent to $\cal G$, such that $F^a$($F^b$) is 
  not an $a$($b$)-fibration for $\cal H$. In fact, the object level map
$F^a_0$($F^b_0)$ for $\cal H$ is not even a quasifibration.\end{thm}

\subsection{$PB_n^a({\cal G})$, $PB_n^b({\cal G})$, $F^a$ and $F^b$}\label{PF}

We have already defined the source ($s_n$) and the target ($t_n$) maps for
$PB_n^a({\cal G})$ in Section \ref{sor},
and observed they are smooth. Now, we define the other structure maps and show that they
are smooth and satisfy the conditions in $\bf C$ of the definition of a groupoid. Recall that
$$PB^a_n({\cal G})_1:=PB^a_{n-1}({\cal G})_1\times_{PB_{n-1}({\cal G})_0}PB_n({\cal G})_0.$$

Since $PB^a_n({\cal G})_1={\cal G}_1$, we again use induction to define the other
maps. So, assume we have defined the inverse, multiplication and the unit maps for
$PB^a_{n-1}({\cal G})_1$ and they are smooth, and satisfies the conditions in $\bf C$.

Let $\alpha_n=(\alpha_{n-1}, x)\in PB^a_n({\cal G})_1$. Define $i(\alpha_n)=(\alpha_{n-1}^{-1}, x)$.
Next, define $u(x)=(id_{(x_1,x_2,\ldots ,x_{n-1})}, x)$ for $x=(x_1,x_2,\ldots ,x_n)\in PB_n({\cal G})_0$.

Recall that $s_n(\alpha_n)=x$ and $t_n(\alpha_n)=(t_{n-1}(\alpha_{n-1}), x_n)$. Let
$\alpha_n'\in PB^a_n({\cal G})_1$, such that $t_n(\alpha_n)=s_n(\alpha_n')$. This
implies $(t_{n-1}(\alpha_{n-1}), x_n)=x'=(x_1',x_2',\ldots ,x_n')$. Therefore,
$t_{n-1}(\alpha_{n-1})=(x_1',x_2',\ldots ,x_{n-1}')=s_{n-1}(\alpha_{n-1}')$. By induction,
we define $\alpha'_n\circ \alpha_n=(\alpha'_{n-1}\circ\alpha_{n-1}, x)$.

Clearly, all the maps defined above are smooth. The checking of the conditions
in $\bf C$ are straight forward.

Now we consider the case $PB_n^b({\cal G})$. Recall that,
$$PB_n^b({\cal G})_1:=(s^n, t^n)^{-1}(PB_n({\cal G})_0\times PB_n({\cal G})_0).$$

The structure maps in this case are easily defined.
Let $\alpha=(\alpha_1,\alpha_2,\ldots , \alpha_n)\in PB_n^b({\cal G})_1$ and
$x=(x_1,x_2,\ldots , x_n)\in PB_n({\cal G})_0$, then, define
the source, target, inverse and unit maps as follows.
$$s(\alpha)=(s(\alpha_1),s(\alpha_2),\ldots , s(\alpha_n)),
t(\alpha)=(t(\alpha_1),t(\alpha_2),\ldots , t(\alpha_n)),$$
$$i(\alpha)=(\alpha_1^{-1},\alpha_2^{-1},\ldots , \alpha_n^{-1}),
u(x)=(id_{x_1},id_{x_2},\ldots , id_{x_n}).$$ If
$\alpha'=(\alpha_1',\alpha_2',\ldots , \alpha_n')\in PB_n^b({\cal G})_1$ with
$t(\alpha)=s(\alpha')$ then $t(\alpha_i)=s(\alpha'_i)$, for all $i=1,2,\ldots, n$ and
hence we can define the multiplication as follows.
$$\alpha'\circ \alpha=(\alpha'_1\circ\alpha_1,\alpha'_2\circ\alpha_2,\ldots , \alpha'_n\circ\alpha_n).$$
Since $s(\alpha'_i\circ\alpha_i)=s(\alpha_i)$ and $t(\alpha'_i\circ\alpha_i)=t(\alpha_i')$, for
all $i=1,2,\ldots, n$, and no two of $s(\alpha_i)$ or of $t(\alpha_i')$ have the same
orbit, $\alpha'\circ\alpha$ is well-defined.

Finally, all of these maps are smooth and satisfy the conditions in $\bf C$.

Next, we prove Lemmas \ref{pb} and \ref{pb1}.

\begin{proof}[Proofs of Lemmas \ref{pb} and \ref{pb1}]
  Recall that we need to proof that $F^a$ and $F^b$ are homomorphisms of
  Lie groupoids. That is, they are smooth functors and commute with the
  source, target, inverse,  unit and multiplication maps of
  the domain and range groupoids. This means the following.

  Let $f:{\cal K}\to {\cal L}$ be a smooth functor between two Lie groupoids. Assume
  $f$ is defined by two
  maps $f_0:{\cal K}_0\to {\cal L}_0$ and $f_1:{\cal K}_1\to {\cal L}_1$ on the
  object and morphism spaces. We denote the structure maps of $\cal K$ and
  $\cal L$ by the same notations.
  $f$ is called a {\it homomorphism} if 
  $f_0$ and $f_1$ commute with the structure maps on the domain and the range
  Lie groupoids. That is, the following are satisfied. $$(a).\ s\circ f_1=f_0\circ s,
  (b).\ t\circ f_1=f_o\circ t, (c).\ f_1\circ u=u\circ f_0,$$
  $$ (d).\ f_1\circ i=i\circ f_1, (e).\ f_1\circ m=m\circ (f_1\times f_1).$$
  Here, $f_1\times f_1$ denotes the induced map
  ${\cal K}_1\times_{{\cal K}_0}{\cal K}_1\to {\cal L}_1\times_{{\cal L}_0}{\cal L}_1$, using $(a)$ and $(b)$.
  
  Clearly, $F^a$ and $F^b$ are smooth functors. Showing the properties
  in the above display, for them are straight forward, nevertheless, we check
  it for $F^a$ and leave the $F^b$ case for the reader.

  We check the equations $(a)$ to $(e)$ for $F^a$.
  Recall that we denoted the source and the target maps of $PB_n^a({\cal G})$ by
  $s_n$ and $t_n$ respectively. Let $n\geq 2$. Let $\alpha_n=(\alpha_{n-1}, x)\in PB^a_n({\cal G})_1$.
  Then $F^a_1(\alpha_n)=\alpha_{n-1}$, $t_n(\alpha_n)=(t_{n-1}(\alpha_{n-1}), x_n)$,
  $s_n(\alpha_n)=x$, $u(x)=id_x=(id_{(x_1,x_2,\ldots ,x_{n-1})}, x)$, $i(\alpha_n)=(\alpha_{n-1}^{-1}, x)$.

  $(a).$ We have $(s_{n-1}\circ F^a_1)(\alpha_n)=s_{n-1}(\alpha_{n-1})=(x_1,x_2,\ldots, x_{n-1})$, and 
$(F^a_0\circ s_n)(\alpha_n)=F^a_0(x)=(x_1,x_2,\ldots, x_{n-1}).$

  $(b).$ Note that, $(t_{n-1}\circ F^a_1)(\alpha_n)=t_{n-1}(\alpha_{n-1})$. On the
  other hand $(F^a_0\circ t_n)(\alpha_n)=F^a_0(t_{n-1}(\alpha_{n-1}), x_n)=t_{n-1}(\alpha_{n-1}).$

  $(c).$ $(F^a_1\circ u)(x)=F^a_1(id_{(x_1,x_2,\ldots ,x_n)})=
  F^a_1(id_{(x_1,x_2,\ldots ,x_{n-1})}, x)=id_{(x_1,x_2,\ldots ,x_{n-1})}$, and
  $(u\circ F^a_0)(x)=u(x_1,x_2,\ldots ,x_{n-1})=id_{(x_1,x_2,\ldots ,x_{n-1})}$.

  $(d).$ $(F^a_1\circ i)(\alpha_n)=F^a_1(\alpha_{n-1}^{-1}, x)=\alpha_{n-1}^{-1}$. Next,
  $(i\circ F^a_1)(\alpha_n)=i(\alpha_{n-1})=\alpha_{n-1}^{-1}$.

  $(e).$ Let $\alpha'_n\in PB^a_n({\cal G})$ such that $t_n(\alpha_n)=s_n(\alpha'_n)$, that is
  $(t_{n-1}(\alpha_{n-1}), x_n)=x'$. Then, $$(m\circ(F^a_1\times F^a_1))(\alpha_n, \alpha'_n)=m(\alpha_{n-1},\alpha_{n-1}')
  =(\alpha'_{n-2}\circ \alpha_{n-2}, (x_1,x_2,\ldots ,x_{n-1}))$$
  and $$(F^a_1\circ m)(\alpha_n, \alpha'_n)=F^a_1(\alpha'_{n-1}\circ \alpha_{n-1}, x)=\alpha'_{n-1}\circ\alpha_{n-1}$$$$=
  (\alpha'_{n-2}\circ \alpha_{n-2}, (x_1,x_2,\ldots ,x_{n-1})).$$

  This completes the proof that $F^a$ is a homomorphism, that is, the proof of Lemma \ref{pb} is complete.
  As we mentioned before that the proof of Lemma \ref{pb1} is similar.\end{proof}

\begin{exm}{\rm If $\cal G$ is an orbifold groupoid, then the configuration
    Lie groupoids $PB^a_n({\cal G})$ and $PB^b_n({\cal G})$ are also orbifold groupoids.
    But for the same orbifold groupoid $\cal G$, the above two configuration Lie groupoids, although
    have the same object space, they define orbifold groupoid structures on different
    orbifolds. $|PB^a_n({\cal G})|$ is a larger space than $|PB^b_n({\cal G})|$. In fact,
    there is a homomorphism $PB^a_n({\cal G})\to PB^b_n({\cal G})$ which is identity on the
    object space and sends $(\alpha, (x_1,x_2),\ldots, (x_1,x_2,\ldots, x_n))\in PB_n^a({\cal G})_1$ to
    $(\alpha,id_{x_2},id_{x_3},\ldots, id_{x_n})\in PB_n^b({\cal G})_1$.}\end{exm}

\section{Orbifold braid groups} \label{obg}
In this section we give a short introduction to orbifold braid groups.
We also use a stretching technique and prove few basic results
on the orbifold braid group of the
orbifold ${\Bbb C}(k,m; q)$,   
which are needed to prove Theorem \ref{esg}.

We have already defined the pure orbifold braid group of $n$ strings
of ${\Bbb C}(k,m; q)$, as the orbifold fundamental group of the configuration
orbifold (Definition \ref{pobg}). Since the underlying
space of ${\Bbb C}(k,m;q)$ is an open subset of $\Bbb C$, there is one more 
way one can define the (pure) orbifold braid group, which give the same result (see [\cite{All}, p. 3]).
It is the pictorial way as in the classical
braid group case.

The later pictorial definition is relevant for us. We describe it now
from \cite{All}.

\medskip

\centerline{\includegraphics[height=4cm,width=8cm,keepaspectratio]{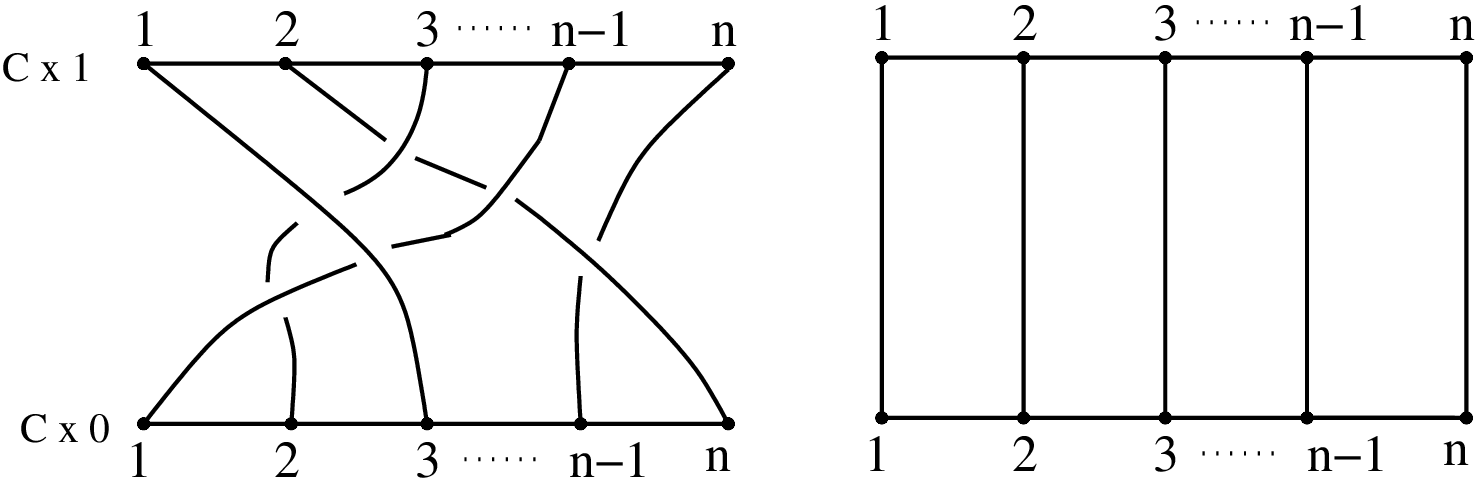}}

\centerline{Figure 1: A typical braid and the identity element.}

\medskip

Recall that any element of the classical braid group 
$\pi_1(B_n({\Bbb C}))$ 
is identified with an equivalence class of a  
braid. An example of a braid is given in the first picture of Figure 1. Two braids are
called {\it equivalent}, if one can be obtained from the other by moving the
strings, fixing the end points, such that, in the process no two strings touch or
cross each other. Juxtaposing one braid over another gives 
the group operation. And, the identity element is the braid which 
joins the vertex $j$ to $j$, for $j=1,2,\ldots, n$, and no two strings entangle with each other, as in
the second picture of Figure 1. See \cite{Ar}.

Consider the complex plane with only one cone 
point $x$, that is ${\Bbb C}(0,1;q)$. The underlying topological space of
${\Bbb C}(0,1;q)$ is nothing but the complex plane. 
Therefore, $B_n({\Bbb C}(0,1;q))$ is an orbifold when we consider the orbifold structure 
of ${\Bbb C}(0,1;q)$, otherwise it is the classical braid space
$B_n({\Bbb C})$. 

\medskip

\centerline{\includegraphics[height=4cm,width=8cm,keepaspectratio]{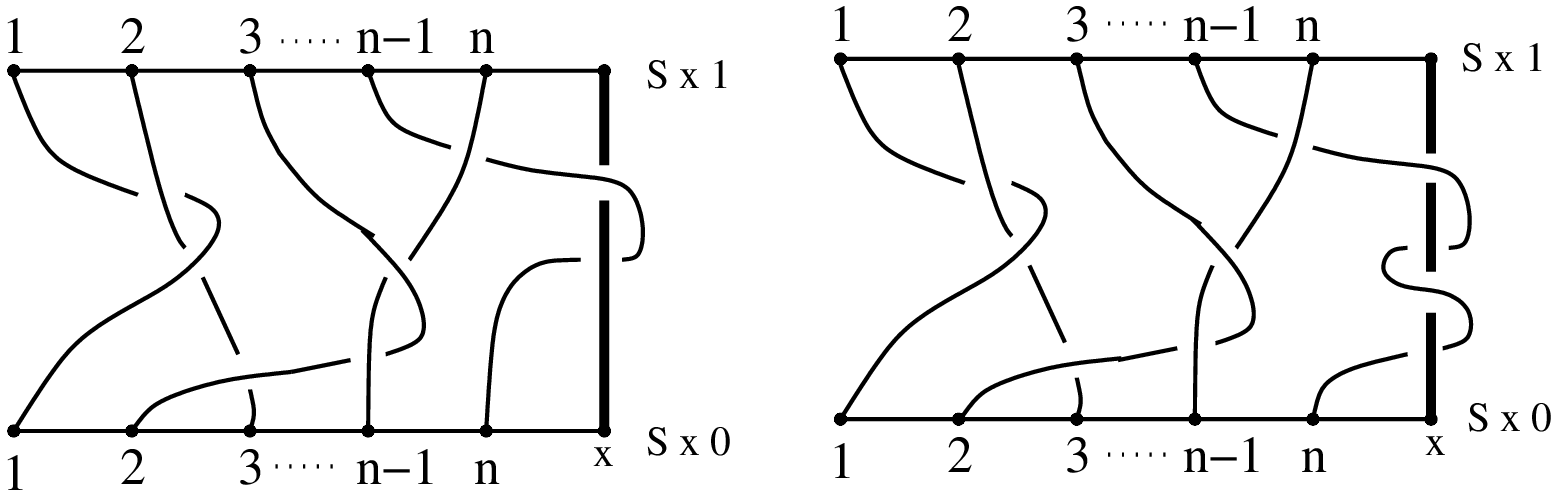}}

\centerline{Figure 2: Orbifold braid.}

\medskip

Therefore, although the fundamental group 
of the underlying topological space of $B_n({\Bbb C}(0,1;q))$ has the classical 
braid representation as above, the orbifold fundamental group 
of $B_n({\Bbb C}(0,1;q))$ needs a different treatment.
We point out here a similar braid representation 
of the orbifold fundamental group of $B_n({\Bbb C}(0,1;q))$
from \cite{All}. The pictures in Figure 2 above   
shows the case of one cone point $x$. The thick line 
represents $x\times I$. Here $S={\Bbb C}(0,1;q)$.

Note that both the braids in Figure 2 represent the same element in
$\pi_1(B_n({\Bbb C}))$ (in this case there is no thick line), but 
different in $\pi_1^{orb}(B_n({\Bbb C}(0,1;q)))$, depending on the order of the cone point.
We describe it below.

One has to define new relations among braids, 
respecting the cone points of the orbifold. We produce one situation 
to see how this is done. 
The second picture in Figure 3 represents part of a typical element.

\medskip

\centerline{\includegraphics[height=9cm,width=11cm,keepaspectratio]{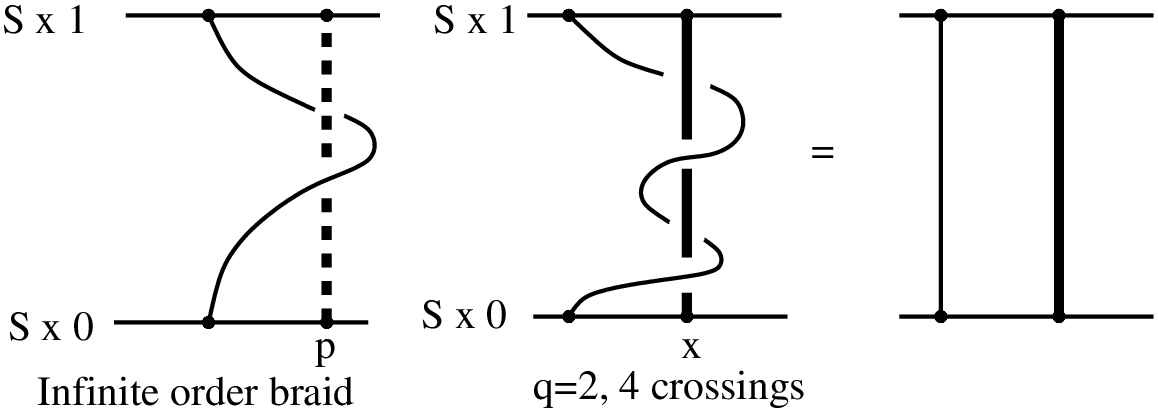}}

\centerline{Figure 3: Movement around a puncture and a cone point.}

\medskip

Now, if a string in the braid wraps the thick line $x\times I$, $q$ times (that is, $2q$ crossings), then   
it is equal to the third picture. This is because, if a loop circles 
$q$ times around the cone point $x$, then the loop gives the 
trivial element in the orbifold fundamental group of ${\Bbb C}(0,1; q)$. Therefore, both braids represent  
the same element in the orbifold fundamental group of $B_n({\Bbb C}(0,1; q))$. Furthermore, if the 
string wraps the thick line, not in a multiple of $q$ number of times, then it is 
not equal to the unwrapped braid. For more details 
see \cite{All}.

Let us now consider the orbifold ${\Bbb C}(1,0;q)$, that is the punctured complex plane.
When there is a puncture $p$, then the
braids will have to satisfy a similar property, but in this case if
any of the string wraps $p\times I$ at least once, then the braid will have infinite order. See
the first picture in Figure 3. One can also think of $p\times I$ as a
fixed string (see Remark \ref{p}).

We now consider the general case. Recall that ${\Bbb C}(k,m;q)$ is
the orbifold, whose underlying space is
${\Bbb C}-\{p_1,p_2,\ldots, p_k\}$, with cone points at $x_i\in {\Bbb C}-\{p_1,p_2,\ldots, p_k\}$ of
order $q_i$ for $i=1,2,\ldots, m$. Then, we have the following exact sequence, since the
quotient map $PB_n({\Bbb C}(k,m;q))\to B_n({\Bbb C}(k,m;q))$ is an orbifold covering map, with
$S_n$ as the group of covering transformation. 

\centerline{
\xymatrix{1\ar[r]&\pi_1^{orb}(PB_n({\Bbb C}(k,m;q)))\ar[r]&\pi_1^{orb}(B_n({\Bbb C}(k,m;q)))\ar[r]&S_n\ar[r]&1.}}
\noindent
Therefore, the elements of $\pi^{orb}_1(PB_n({\Bbb C}(k,m:q)))$ are braids, where the
strings join $j$ to $j$ for $j=1,2,\ldots,n$.

A typical
element of $\pi_1^{orb}(PB_n({\Bbb C}(k,m;q)))$ looks like $A$ as in
the following figure. Here, note that we have placed the punctures and cone points
conveniently on the right hand side. This does not change the group, because the
topology of the orbifold ${\Bbb C}(k,m;q)$ is independent of the position of the
punctures or the cone points. The group operation is again given by juxtaposing
one orbifold braid onto another, and the identity element is also obvious.

\medskip

\centerline{\includegraphics[height=4cm,width=8cm,keepaspectratio]{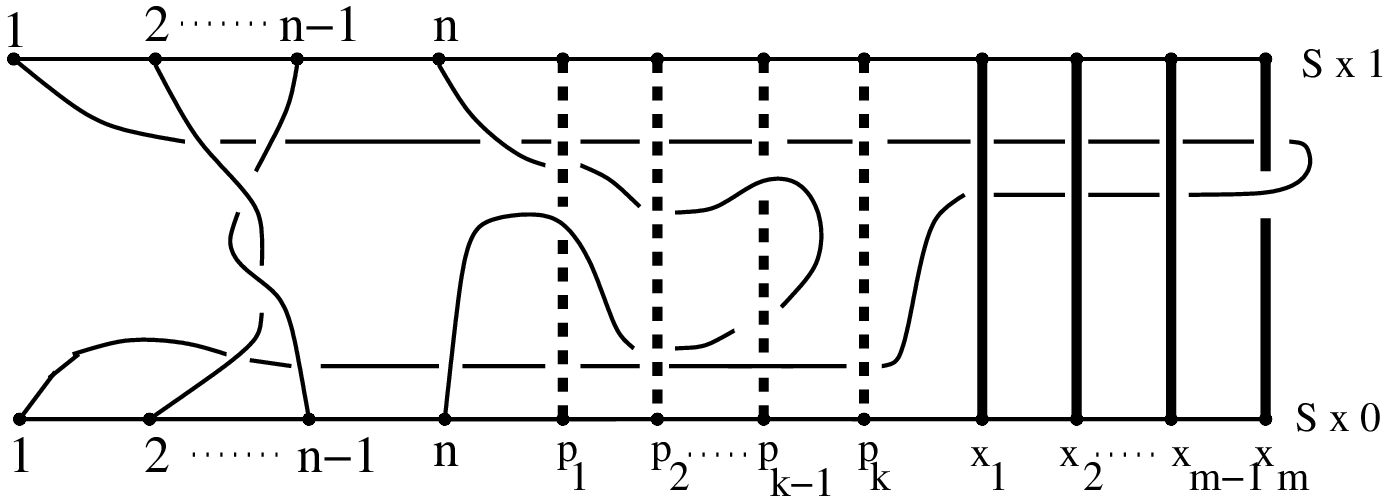}}

\centerline{Figure 4: An element $A$ in $\pi_1^{orb}(PB_n({\Bbb C}(k,m; q)))$.}

\medskip

Now, we describe a set of generators for
$\pi_1^{orb}(PB_n({\Bbb C}(k,m; q)))$.

Recall from \cite{Ar}, that $\pi_1(B_n({\Bbb C}))$
is generated by $\sigma_i$, $i=1,2,\ldots, n-1$, and 
$\pi_1(PB_n({\Bbb C}))$ 
is generated by the braids $B^{(n)}_{ij}$, $i<j$ as shown below. Here,
the string from $i$ to $i$ is going below all the strings up to the
string $j-1$ to $j-1$.

\medskip

\centerline{\includegraphics[height=3cm,width=6cm,keepaspectratio]{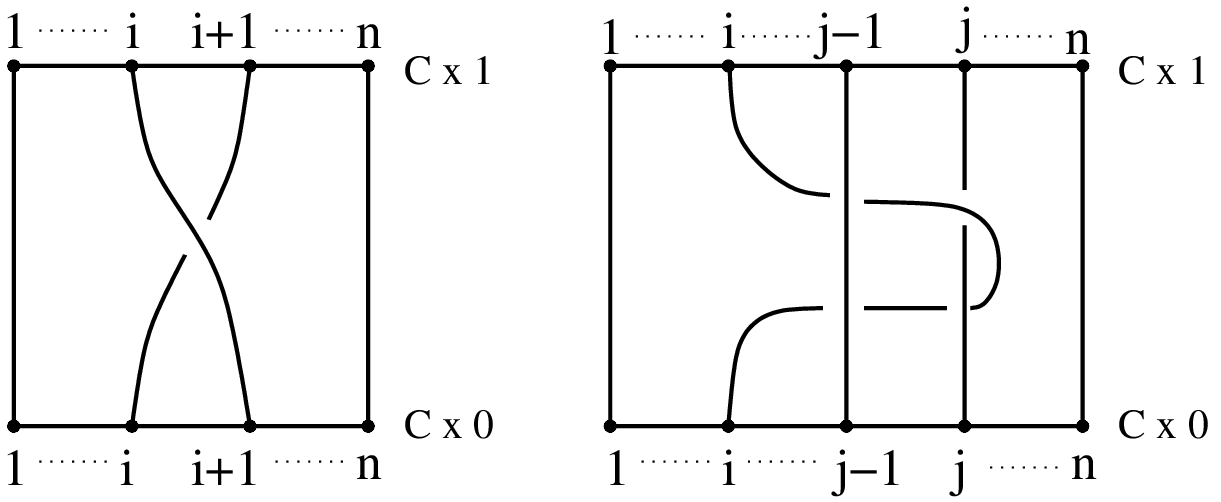}}

\centerline{Figure 5: The generator $\sigma_i$ of $\pi_1(B_n({\Bbb C}))$
  and $B^{(n)}_{ij}$ of $\pi_1(PB_n({\Bbb C}))$.}

\medskip

A quick drawing shows the following.
$$B^{(n)}_{ij}=\sigma_{j-1}\sigma_{j-2}\cdots \sigma_{i+1}\sigma_i^2\sigma_{i+1}^{-1}\cdots \sigma_{j-2}^{-1}\sigma_{j-1}^{-1}.$$

\begin{lemma}\label{orbi-gen} A set of generators for $\pi_1^{orb}(PB_n({\Bbb C}(k,m; q)))$ is given in
Figures 6,7 and 8. 

\medskip

\centerline{\includegraphics[height=2.5cm,width=5cm,keepaspectratio]{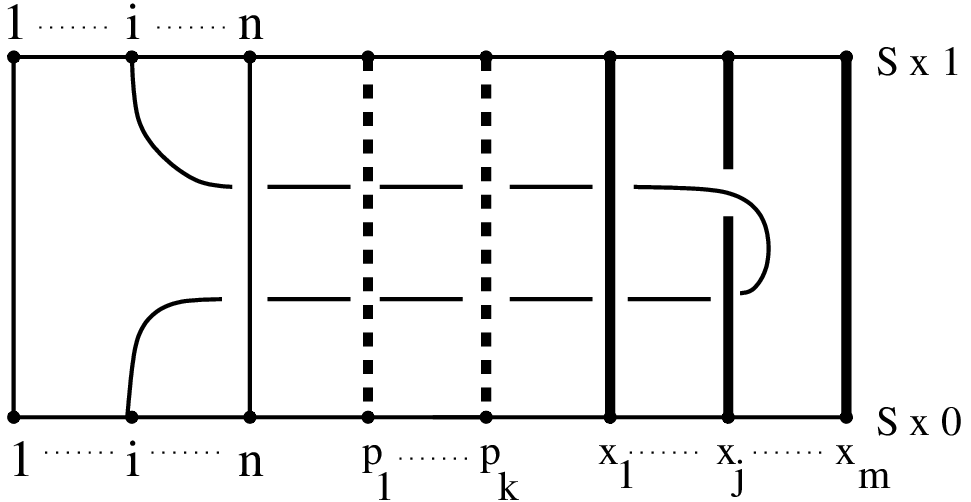}
  \includegraphics[height=2.5cm,width=5cm,keepaspectratio]{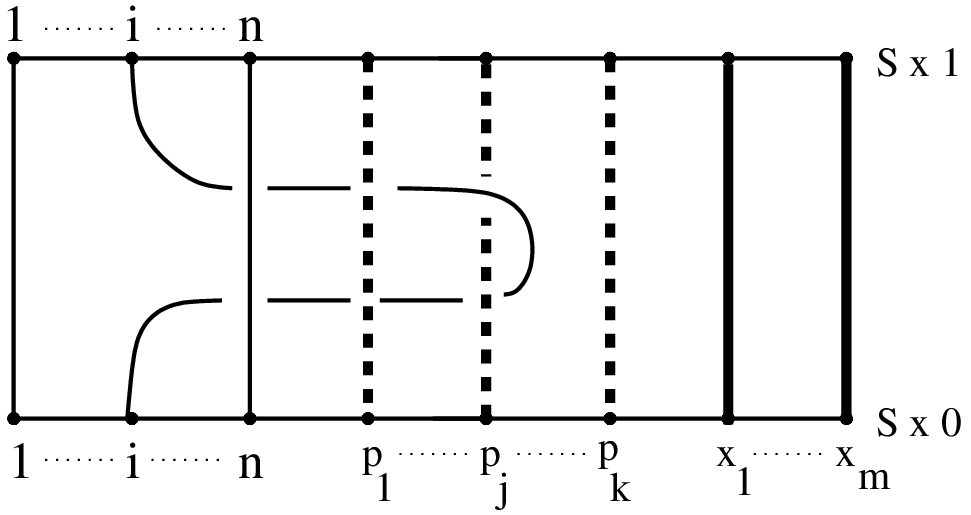}}

\centerline{Figure 6: The generator $X_{ij}$\ \ \ \ \ \ Figure 7: The generator $P_{ij}$.}

\medskip

\centerline{\includegraphics[height=5cm,width=10cm,keepaspectratio]{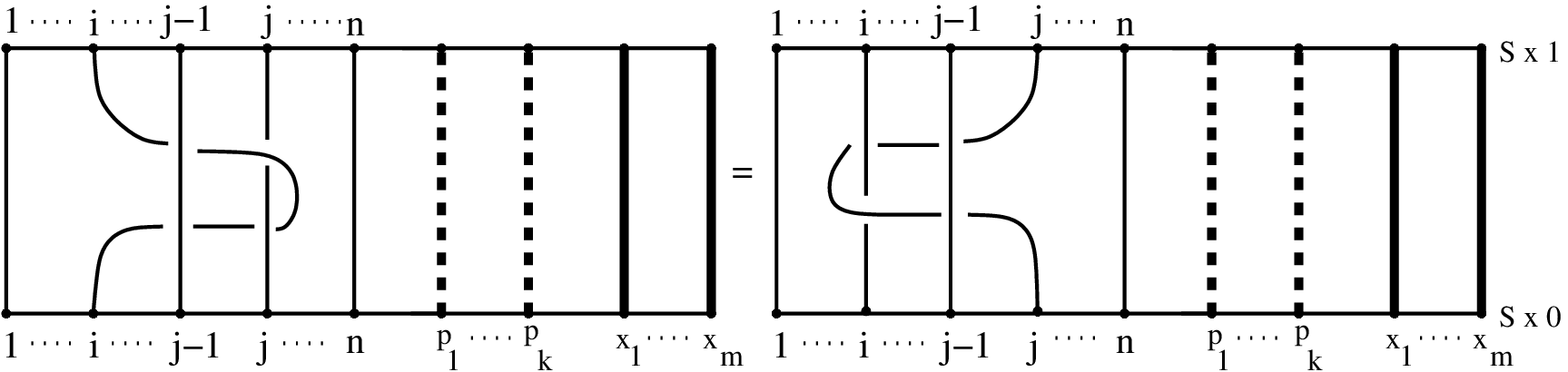}}

\centerline{Figure 8: The generator $B_{ij}$.}\end{lemma}

\begin{proof} Consider the classical pure braid group
  $\pi_1(PB_{n+k+m}({\Bbb C}))$ and its generators
  $B^{(n+k+m)}_{ij}$, $i=1,2,\cdots, n+k+m-1; j=2,\cdots, n+k+m,
  i<j$, as recalled above. Let $G$ be the subgroup of $\pi_1(PB_{n+k+m}({\Bbb C}))$
  generated by $B^{(n+k+m)}_{ij}$, $i=1,2,\cdots, n; j=2,\cdots,
  n+k+m; i<j$. Clearly, any braid $B$ representing an element of $G$ is equivalent to
  a braid, whose 
  all the last $k+m$ strings are vertically straight and not entangling
  with each other.
  
  Now, we replace these last $k+m$ straight strings by dotted
  lines and denote the corresponding braid by $\bar B$. That is, we 
  introduced $k+m$ punctures in $\Bbb C$. Then, clearly we have an isomorphism from $G$
  onto $\pi_1(PB_n({\Bbb C}(k+m,0;q)))$, sending $B$ to $\bar B$ (see [\cite{G-M}, page 26]).
  Since, the compositions in
  $G$ and in $\pi_1(PB_n({\Bbb C}(k+m,0;q)))$ are same, that is,
  juxtaposition of braids. Furthermore, 
  considering the group operations in $G$ and in $\pi_1(PB_n({\Bbb C}(k+m,0;q)))$,
  there is no difference between a dotted line and a straight
  string.

  Hence, the
  relations in a presentation of $G$ with respect to the above set of
  generators, and the relations in a
  presentation of $\pi_1(PB_n({\Bbb C}(k+m,0;q)))$, in terms of the generators 
$\bar B^{(n+k+m)}_{ij}$, $i=1,2,\cdots, n; j=2,\cdots, n+k+m; i<j$ are
identical, except with a bar.

Next, we replace the last $m$ dotted lines in $\bar B$, by thick
lines, and denote it by $\bar {\bar B}$. That is, we filled the
last $m$ punctures by orbifold points of order
$q_i, i=1,2,\cdots, m$. Then, the element 
$\bar{\bar B}^{(n+k+m)}_{ij}, j=n+k+1,\cdots ,n+k+m$ has order
$q_{j-n-k}$, as we described the orbifold braid groups before, from \cite{All}. Hence, we
get back the group $\pi_1(PB_n({\Bbb C}(k,m;q)))$, and its generators
are as described in the statement of the Lemma. Since, clearly any
element of $\pi_1(PB_n({\Bbb C}(k,m;q)))$ is of the form $\bar{\bar
  B}$, for some $B\in G$.  \end{proof}

\begin{rem}\label{G}{\rm We denote by $${\cal O}_n:G\simeq \pi_1(PB_n({\Bbb C}(k+m,0;q)))\to
    \pi_1^{orb}(PB_n({\Bbb C}(k,m; q))),$$ the surjective
    homomorphism sending $B$ to $\bar{\bar {B}}$.}\end{rem}

\begin{rem}\label{p}{\rm From the above proof we also get that the pure orbifold braid
    group of $n$ strings of ${\Bbb C}(k,m;q)$, can be
    embedded into the pure orbifold braid group of $n+k$
    strings of ${\Bbb C}(0,m;q)$. For $k=m=1$, this embedding was proved in
    [\cite{Rou}, Proposition 4.1].}\end{rem}

Now we come to the crucial lemma, which is the main
ingredient for this paper.

\begin{lemma}\label{stretching} Let $A\in \pi_1^{orb}(PB_n({\Bbb C}(k,m;
  q)))$ satisfying the following property.

  $\diamond$ $A$ is equivalent to a braid which has all the first $n-1$
  strings not entangling with each other, they are vertically straight
  and only the string from $n$ to $n$ (say $N$)
is entangling with some (or all) of the first $n-1$ strings or the dotted or
the thick lines.

Then, $A$ is equivalent to a juxtaposition of the following braids or
their inverses.

$$X_{nr}, r=1,2,\ldots ,m;\ P_{ns}, s=1,2,\ldots , k;\ B_{in},
i=1,2,\ldots , n-1.$$
\end{lemma}

\begin{rem}{\rm Clearly, the braid $A$ in Lemma \ref{stretching} lies in the
    kernel of the homomorphism $\Delta:\pi_1^{orb}(PB_n({\Bbb C}(k,m;
    q)))\to \pi_1^{orb}(PB_{n-1}({\Bbb C}(k,m;
    q)))$, where $\Delta$ sends a braid of $n$ strings to the braid
    of $n-1$ strings, after removing the string from $n$ to $n$.
    We will be using the lemma in cases 
    where $\diamond$ is easily seen by some simple movements of a
    strings.}\end{rem}
    
\begin{proof}[Proof of Lemma \ref{stretching}]
  The proof is basically an application of Lemma \ref{orbi-gen} and Remark \ref{p}.

  First consider the punctured complex plane ${\Bbb C}(k+m+n-1,0;0)$. Then
  by Lemma \ref{orbi-gen} the free group
  $$\pi_1({\Bbb C}(k+m+n-1,0;0))=\pi_1(PB_1({\Bbb C}(k+m+n-1,0;0)))$$ is generated
  by the braids of one string as in the first picture of Figure 9.
  Since positions of the punctures do not
  affect the fundamental group, we move the first $n-1$ dotted lines
  to the left of the string as in the second picture in Figure 9.
  The generators which wraps these first $n-1$ dotted lines also
  are similarly drawn using allowable moves.

 \medskip

\centerline{\includegraphics[height=5cm,width=10cm,keepaspectratio]{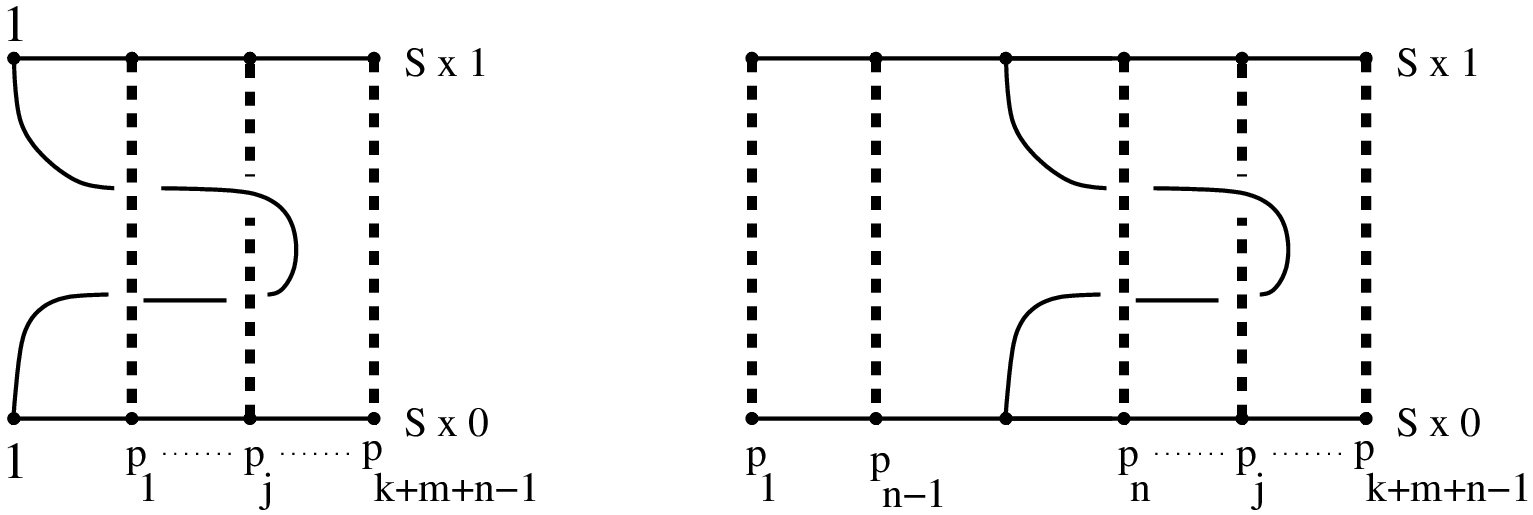}}

\centerline{Figure 9: The generators of $\pi_1(PB_1{\Bbb C}(k+m+n-1,0;0))$.}

\medskip

  Now using Remark \ref{p}, consider the embedding
 $$\pi_1(PB_1({\Bbb C}(k+m+n-1,0;0)))\to \pi_1(PB_n({\Bbb C}(k+m,0;0)))$$ by
  replacing the first $n-1$ dotted lines with straight strings. Next,
  by replacing the last $m$ dotted lines with thick lines, that is, by filling
  the last $m$ punctures with cone points of order
  $q_1, q_2,\ldots, q_m$, we get a homomorphism
  $$Q:\pi_1(PB_1({\Bbb C}(k+m+n-1,0;0)))\to \pi_1^{orb}(PB_n({\Bbb C}(k,m;q))).$$
  Note that, the image of this homomorphism consists of braids exactly
  of the type $A$, satisfying $\diamond$. Choose $\widetilde A\in Q^{-1}(A)$. Then $\widetilde A$ is
  a juxtaposition of generators of the type as in Figure 9 and their
  inverses. Now, take the image of these generators which appear in
  a decomposition of $\widetilde A$, under $Q$, to get the desired
  braids as in the statement of the lemma, to construct $A$.

  This completes the proof of Lemma \ref{stretching}.
\end{proof}

Next, we apply Lemma \ref{stretching} to prove the following
proposition, crucial for this paper.

\begin{prop}\label{gen} The subgroup $H$ generated by the following
  set of braids is normal
  in $\pi_1^{orb}(PB_n({\Bbb C}(k,m; q)))$.
$${\cal N}=\{X_{nr}, r=1,2,\ldots ,m;\ P_{ns}, s=1,2,\ldots , k;\ B_{in},
i=1,2,\ldots , n-1\}.$$\end{prop}

\begin{proof} Recall that, 
$\pi_1^{orb}(PB_n({\Bbb C}(k,m; q)))$ is generated
by the braids as described in Lemma \ref{orbi-gen}. Let us denote this
set of braids by $\cal P$.

Therefore, it is enough to prove that $ZYZ^{-1}\in H$, for
all $Z\in \cal P$ and $Y\in \cal N$. That is, we need to show that
$ZYZ^{-1}$ is a juxtaposition of elements from 
$\cal N$ or their inverses. We only have to consider the following nine cases.

$$(1)\ X_{kl}X_{nr}X_{kl}^{-1}, (2)\ X_{kl}P_{ns}X_{kl}^{-1},
(3)\ X_{kl}B_{in}X_{kl}^{-1},$$
$$ (4)\ P_{kl}X_{nr}P_{kl}^{-1}
, (5)\ P_{kl}P_{ns}P_{kl}^{-1}, (6)\ P_{kl}B_{in}P_{kl}^{-1},$$
$$ (7)\ B_{lk}X_{nr}B_{lk}^{-1}
, (8)\ B_{lk}P_{ns}B_{lk}^{-1}, (9)\ B_{lk}B_{in}B_{lk}^{-1}.$$

Therefore, it is enough to show that all the above nine cases satisfy the
condition $\diamond$ of Lemma \ref{stretching}. As there are only three
braids in each case, we draw their picture and observe that condition $\diamond$ is
satisfied. We recall that, condition $\diamond$ says that the braid under study
is equivalent to a braid, where all the first $n-1$
strings are vertically straight. The method we use is simply stretch
the string $N$, so that the other strings (maximum three, depending on the
case) become straight.

After this pictorial proof, we will observe a general pattern in the
stretching method. Then we will give an argument using
isotopy, which will be applicable in all the nine cases.

We can partition the above nine cases 
as follows, so that the cases in each partition class need
similar treatment.

$$A=\{(1),(2),(4),(5)\}, B=\{(3),(6)\}, C=\{(7),(8)\}, D=\{(9)\}.$$

Therefore, we give the proof for one case from each
class, and then describe the modifications required in the proof
to prove the other cases.

First, we note that for $k=n$, there is nothing to prove in all the
nine cases. So, we can assume
that $k\neq n$.

\noindent
{\bf Class A.} We give the proof for the Case (1).
There are three possibilities to consider; $l=r$, $l<r$ and
$l>r$.

This deduction is shown in Figure 10.

\medskip

\centerline{\includegraphics[height=7cm,width=9cm,keepaspectratio]{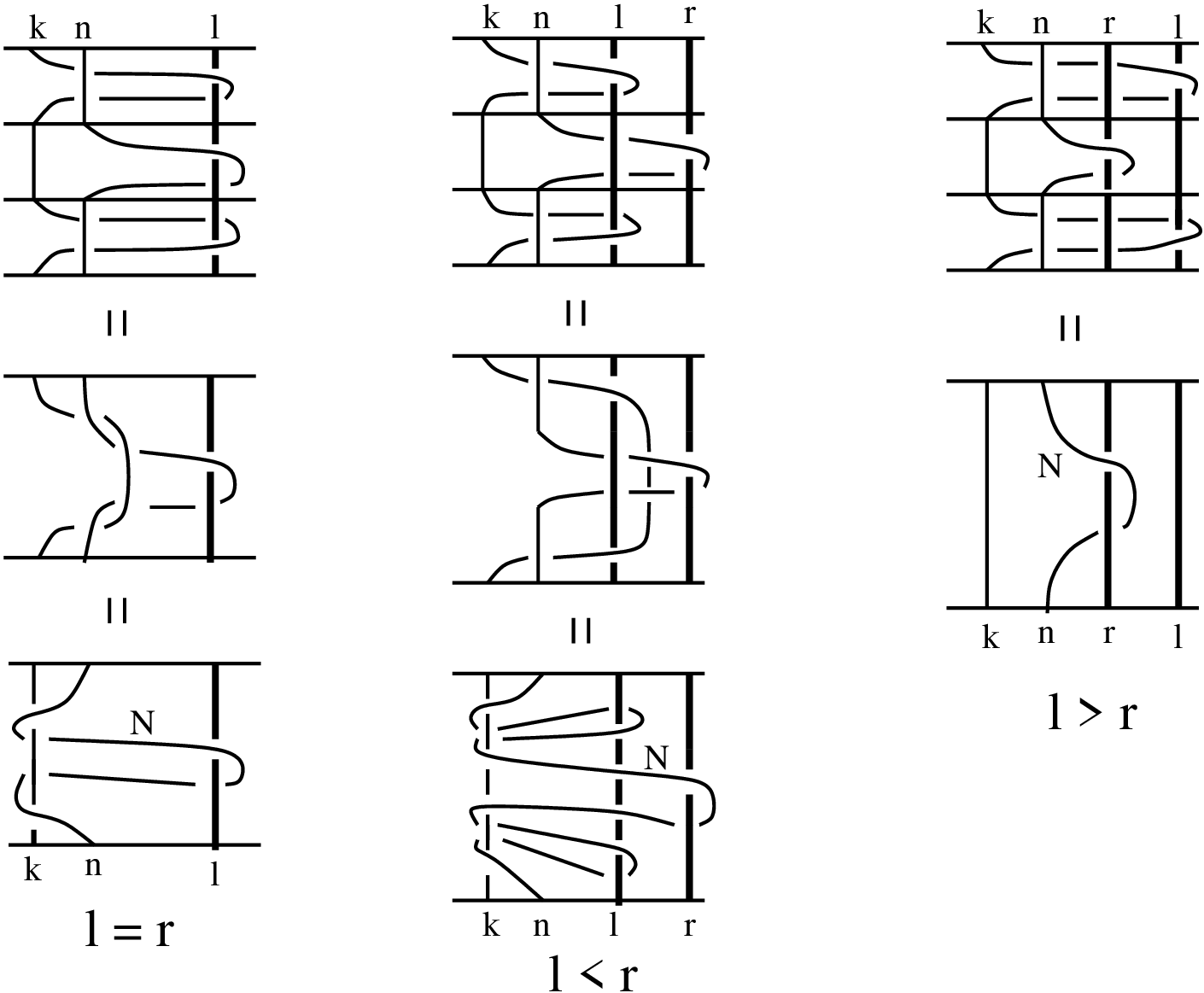}}

\centerline{Figure 10: Case (1), $X_{kl}X_{nr}X^{-1}_{kl}$ in Class A.}

\medskip

Now, we describe the other cases in Class A. Case (5) is exactly the
same as Case (1), we only have to replace the thick lines by dotted 
ones in Figure 10. Next, note that in our presentation of the braids,
the dotted lines appear before the thick lines, and hence, for Case
(2), only the third column in Figure 10 is needed with the first
thick line (corresponding to $r$) converted into a dotted
line. Similarly, for Case (4), the middle column in Figure 10
gives the proof, with the first thick line (corresponding to $l$)
converted into a dotted line.

\noindent
{\bf Class B.} We consider the Case (3) first. In this case also
we have to consider three possibilities:
$i=k$, $i<k$ and $i>k$. This stretching procedure is shown in Figure 11.

The treatment for Case (6) in Class B is exactly the same as
Case (3) as in Figure 11, but we have to replace the thick lines by dotted ones.


\centerline{\includegraphics[height=7cm,width=9cm,keepaspectratio]{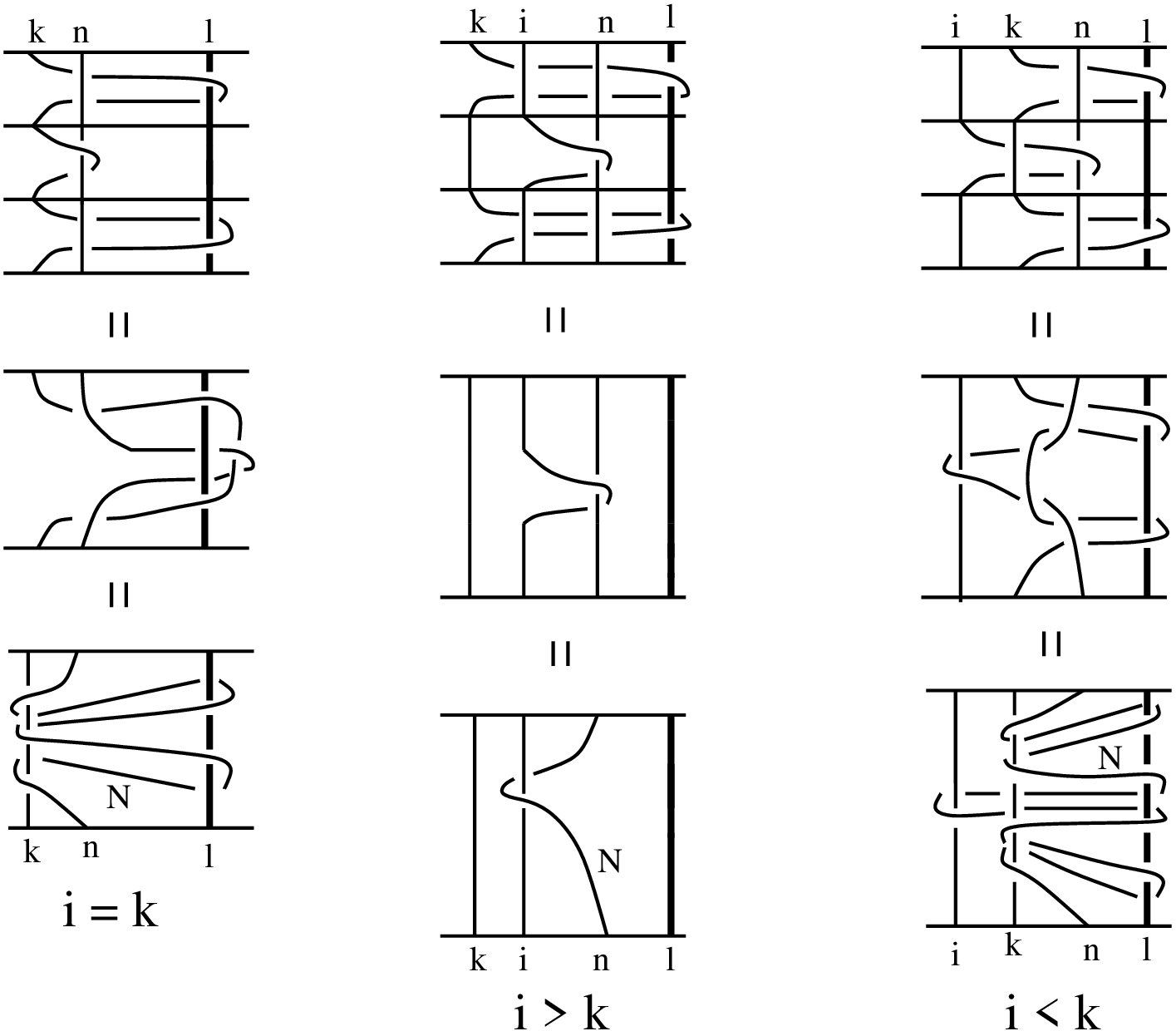}}

\centerline{Figure 11: Case (3), $X_{kl}B_{in}X^{-1}_{kl}$ in Class B.}

\medskip

\noindent
{\bf Class C.} Cases (7) and (8) are the simplest of all the cases. Case (7) is
shown below in Figure 12. Case (8) needs the same treatment as Case (7), with
the thick line replaced by a dotted one.

\medskip

\centerline{\includegraphics[height=3cm,width=6cm,keepaspectratio]{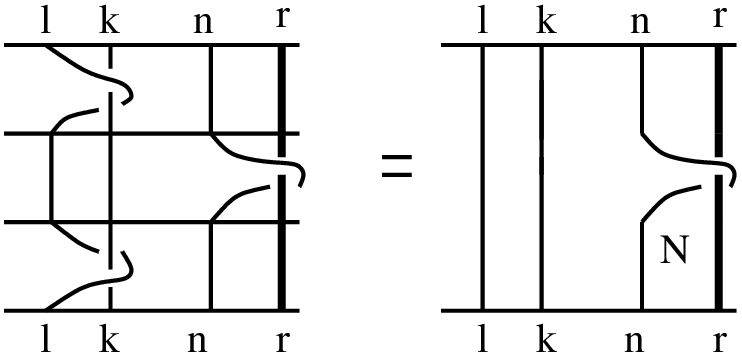}}

\centerline{Figure 12: Case (7), $B_{lk}X_{nr}B^{-1}_{lk}$ in Class C.}

\medskip

\noindent
{\bf Class D.} Case (9) has several possibilities depending on
where $i$ lies: $k<i$, $i<l$, $i=k$, $i=l$ and $l<i<k$. The details are
shown in Figure 13.

One point to note that, for clarity, in the pictures
we did not show the other strings, the dotted lines or the thick 
lines. Since, the movements of the strings are taking place
under all the other strings,
the dotted lines or the thick lines.

Now, we make a general argument which is applicable to all the above
cases.

\centerline{\includegraphics[height=10.5cm,width=10.5cm,keepaspectratio]{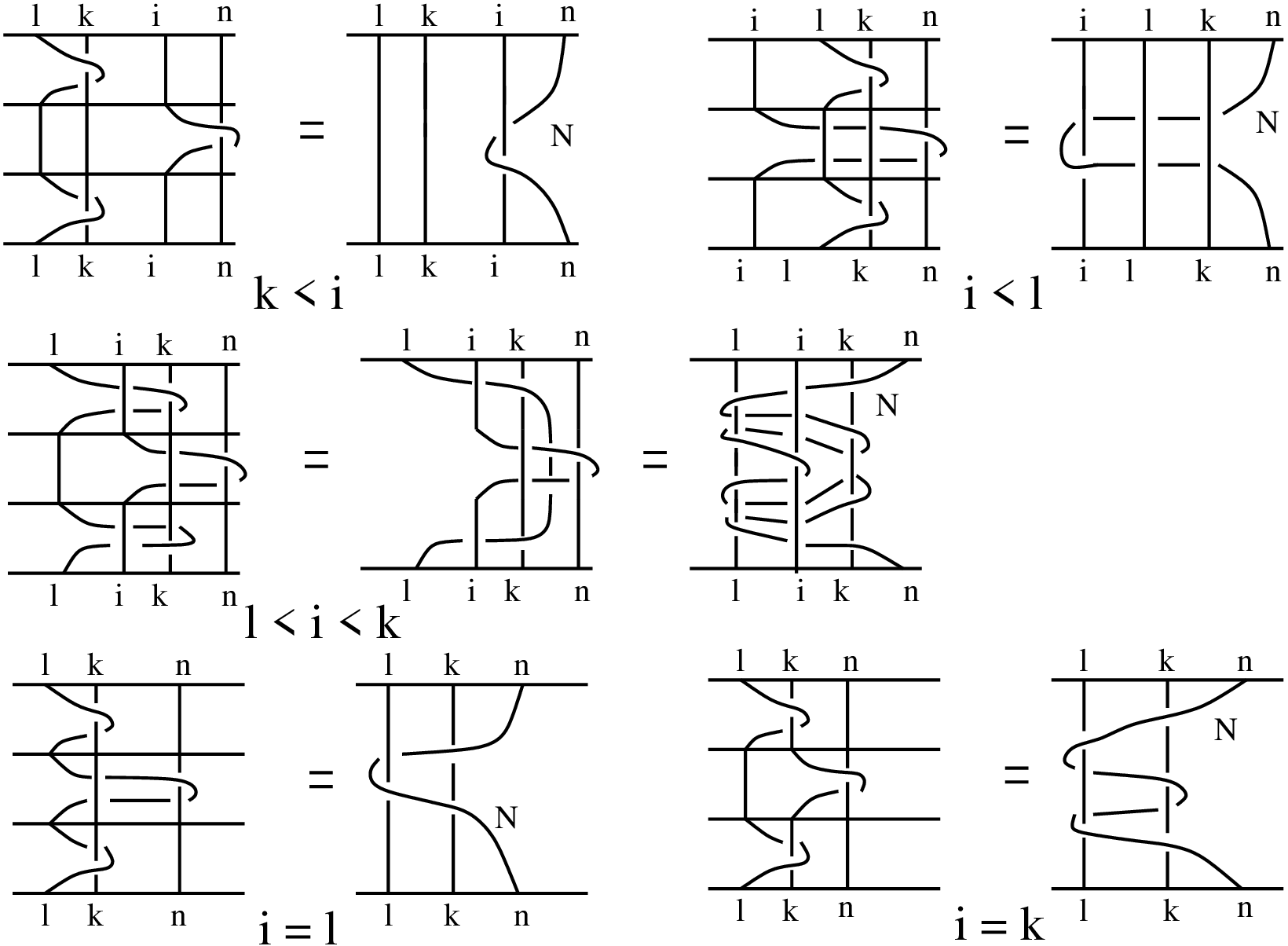}}

\centerline{Figure 13: Case (9), $B_{lk}B_{in}B^{-1}_{lk}$ in Class D.}

\medskip
\noindent
{\bf An observation and the general argument.} Let $B\in ker(\Delta)$. 
First note that, then there is an isotopy of ${\Bbb C}(k+m,0;0)$
which sends $\Delta(B)$ to the identity braid of $PB_{n-1}({\Bbb C}(k,m;q))$.
But some string of $\Delta(B)$ might have had entanglement with the 
thick lines, where the movement as in Figure 3 were used to get the
triviality.

\medskip

\centerline{\includegraphics[height=8cm,width=8cm,keepaspectratio]{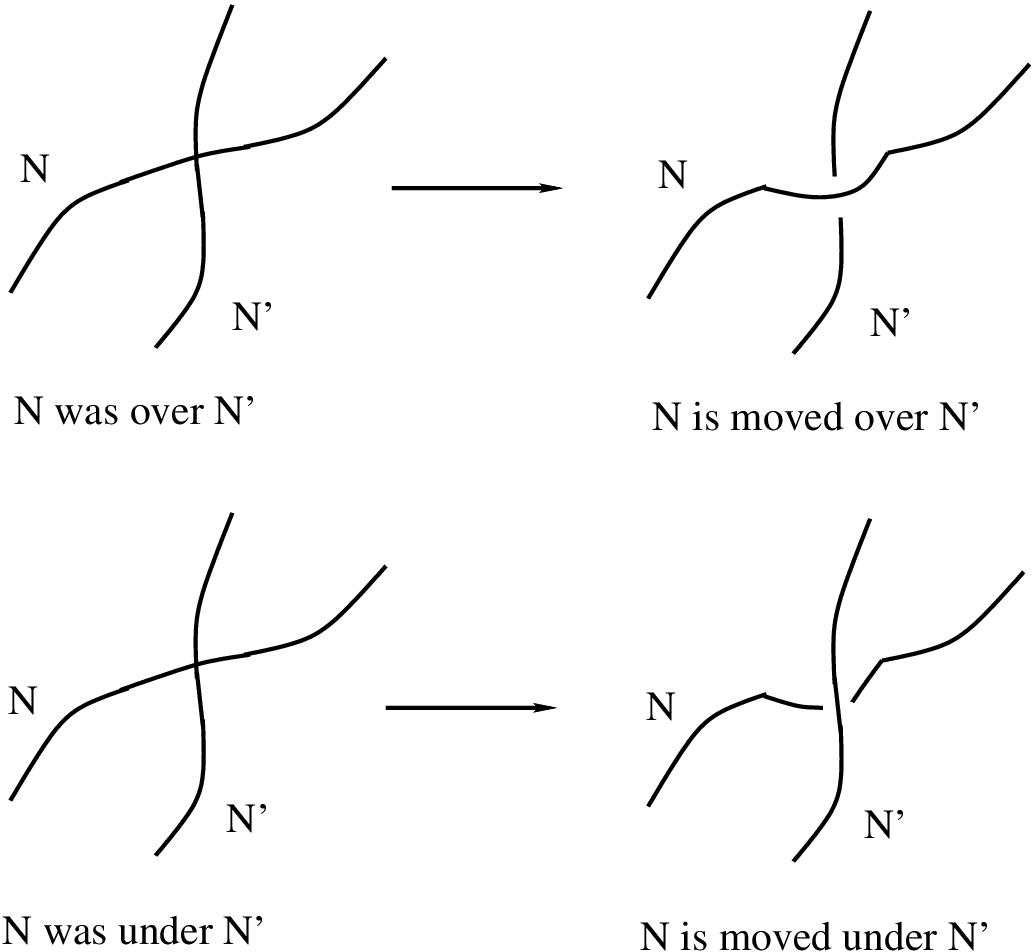}}

\centerline{Figure 14: Local isotopy to remove intersections.}

\medskip
\noindent
But in our nine cases above, it is immediate that
if we remove the string $N$ from any of the orbifold braids, then the other 
strings (maximum number is three) become straight after an
isotopy. That is, the strings other than $N$, are not entangled
with any of the thick lines, that is why we did not need to invoke Figure
3 movement here, which might have complicated matter during the
isotopy.
Therefore, all that we need to do, when $B$ is one of the nine orbifold
braids, is that, we choose an isotopy $T$ to straighten up the other
strings in $B$, and during the isotopy we do not interfere with any of
the other strings except $N$. Now, during the isotopy the string $N$
might have intersected with the other strings. Note that, we can make
sure that there are only finitely many such intersections. At each
intersection point we make one of the two following local isotopies. During the isotopy $T$,
if the string $N$ intersected any
other string, say $N'$, then there are two possibilities. Just before
$N$ intersected $N'$, (a). $N$ was above $N'$ or (b). $N$ was below $N'$. In case
(a), we make the simple local isotopy by moving $N$ away from $N'$ in the
opposite direction, so that there is no intersection in the
neighbourhood of the intersection point. Similarly in case (b) we make
a similar local
isotopy. The local isotopies are shown in Figure 14.
The idea is not to change the equivalence class of $B$ during
the local isotopies.
Hence, $T$ together with these two
local isotopies, give the isotopy which sends the 
orbifold braid $B$ to an orbifold braid, satisfying $\diamond$
of Lemma \ref{stretching}

Therefore,  
the subgroup $H$ generated by the set of braids $\cal N$, 
is a normal subgroup of $\pi_1^{orb}(PB_n({\Bbb C}(k,m; q)))$.

This completes the proof of the proposition.
\end{proof}

We now have an immediate consequence of Proposition \ref{gen} when
restricted to the case $\Bbb C$. This gives another proof of the fact that the kernel of
$f_*$ is a free group, which is a consequence of the
Fadell-Neuwirth fibration theorem. But our proof also produce an
explicit set of generators of this
free group as a subgroup of the classical pure braid group.

\begin{cor}\label{FNF}  The homomorphism $f_*:\pi_1(PB_n({\Bbb
    C}))\to \pi_1(PB_{n-1}({\Bbb C}))$, induced by the
  projection to the first $n-1$ coordinates, $f:PB_n({\Bbb
    C})\to PB_{n-1}({\Bbb C})$, has kernel freely
  generated by the elements $B^{(n)}_{in}$ for $i=1,2,\cdots, 
  n-1$.\end{cor}

\begin{proof} From Case $(9)$ in the proof of Proposition \ref{gen}
  and Lemma \ref{stretching}, the
  corollary follows. We need to specialize to the case ${\Bbb C}$
  only, that is, put $k=m=0$. Then note that the group generated
  by $B^{(n)}_{in}$ for $i=1,2,\cdots, n-1$ is isomorphic to the fundamental group of
  $PB_1({\Bbb C}(n-1,0;0))={\Bbb C}(n-1,0;0)$, which is free.\end{proof}

\section{Proofs}
We start this section with the proof of the Fadell-Neuwirth type
fibration theorem for $c$-groupoids.

We need the following lemma.

\begin{lemma} \label{fib} Let $f:M\to N$ and $g:N\to L$ be surjective maps between
  manifolds. If $g$ is a covering map and $g\circ f$ is a fibration, then $f$
  is also a fibration.\end{lemma}

\begin{proof} By [\cite{Sp}, p. 95-96, Theorems 12, 13], it is enough
  to prove that $f$ is a local fibration, that is, $N$ has a numerable 
  covering by open sets $\{U_i\}_{i\in I}$, so that 
$f|_{f^{-1}(U_i)}:f^{-1}(U_i)\to U_i$ is a fibration
  for all $i\in I$. 

Since $N$ is paracompact, every open covering has a numerable refinement. 
See [\cite{Sp}, p. 95-96]. Therefore, since $g$ is a covering map,
there is a numerable open covering $\{U_i\}_{i\in I}$, such that
for each $i\in I$, $g(U_i)$ is connected, evenly
covered by $g$ and $g|_{U_i}:U_i\to g(U_i)$ is a diffeomorphism. 

Now, showing that 
$f|_{f^{-1}(U_i)}:f^{-1}(U_i)\to U_i$ is a fibration is straight forward.
\end{proof}

\begin{proof}[Proof of Theorem \ref{fibration}]
  Let $\cal G$ be a c-groupoid of dimension $\geq 2$. That is, the
  quotient map $f:{\cal G}_0\to |{\cal G}|$ is a
covering map, and hence $|{\cal G}|$ is again a manifold of dimension $\geq 2$. We have the
following commutative diagram.

\centerline{
\xymatrix{
{\cal G}_0^n\ar[r]\ar[d]^{f^n}& {\cal G}_0^{n-1}\ar[d]^{f^{n-1}}\\
|{\cal G}|^n\ar[r] & |{\cal G}|^{n-1}}}
\noindent
Here, the horizontal maps are projections and the vertical maps are
covering maps. Note that, the diagram
induces the following.

\centerline{
\xymatrix{
PB_n({\cal G})_0\ar[r]\ar[d]^{f^n|_{PB_n({\cal G})_0}}& PB_{n-1}({\cal G})_0\ar[d]^{f^{n-1}|_{PB_{n-1}({\cal G})_0}}\\
PB_n(|{\cal G}|)\ar[r] & PB_{n-1}(|{\cal G}|)}}

By the Fadell-Neuwirth fibration theorem, since $|{\cal G}|$ is a
manifold of dimension $\geq 2$, the lower horizontal map is a
locally trivial fibration. Since the two
vertical maps are covering maps, the top horizontal map is
a fibration of manifolds by Lemma \ref{fib}. 

Now, we check the other conditions of $a$- and $b$-fibrations.

\noindent
$\bf {F^a}.$
To prove that $F^a$ is an $a$-fibration, we only have to show 
that $PB_n({\cal G})_0$ is a (left) $PB^a_{n-1}({\cal G})$-space, with the structure
maps induced from the structure maps of $PB^a_{n-1}({\cal G})$.

We recall the target map from Definition \ref{clg},

$$\mu:PB^a_n({\cal G})_1=PB^a_{n-1}({\cal G})_1\times_{PB_{n-1}({\cal G})_0}PB_n({\cal G})_0\to
PB_n({\cal G})_0$$
of the Lie groupoid $PB^a_n({\cal G})$. This is defined as follows.

Let $(x_1,x_2,\ldots, x_n)\in PB_n({\cal G})_0$ and
$$(\alpha_1,(x_1,x_2),\ldots, (x_1,x_2,\ldots, x_{n-1}))\in PB^a_{n-1}({\cal G})_1.$$
Note that $s(\alpha_1)=x_1$, then define
$$\mu((\alpha_1,(x_1,x_2),\ldots, (x_1,x_2,\ldots, x_{n-1})), (x_1,x_2,\ldots, x_n))=(t(\alpha_1),x_2,\ldots, x_{n-1},x_n).$$

The action map is nothing but the target map of $PB^a_n({\cal G})$.

The remaining properties (see Definition \ref{leftspace}) are easy to check.

Therefore, the homomorphism 
$F^a:PB^a_n({\cal G})\to PB^a_{n-1}({\cal G})$ is an $a$-fibration of Lie groupoids.

\noindent
$\bf {F^b}.$ We now prove that the homomorphism $F^b$ is a $b$-fibration.
For this we only have to check that the map $$f^b:PB^b_n({\cal G})_1\to
PB^b_{n-1}({\cal G})_1\times_{PB_{n-1}({\cal G})_0}PB_n({\cal G})_0$$
is a surjective submersion. But this follows from the following
observations about the different maps involved in the above display.

$(i)$ Let $(\alpha_1, \alpha_2,\ldots,\alpha_n)\in PB^b_n({\cal G})_1$. Then,
this element under the above map goes to
$$((\alpha_1, \alpha_2,\ldots,\alpha_{n-1}), (s(\alpha_1), s(\alpha_2),\ldots,s(\alpha_n))).$$
Hence, the map $f^b$ is obviously surjective, since $s(1_{x_n})=x_n$, and we only need that the last coordinate
in the range has a preimage in ${\cal G}_1$ under $s$.

$(ii)$ Next, note that the spaces involved are open sets in the product space
${\cal G}_0^n$ or ${\cal G}_1^n$ and the maps are either projections or the map $s$ at the
coordinate level. Furthermore, $s$ is a submersion by assumption. Therefore, the map $f^b$
is a submersion.

This completes the proof of the theorem.
\end{proof}

\begin{rem}{\rm A statement similar to Theorem \ref{fibration} also holds
    by taking projection to the first $k$ coordinates, from $PB_n({\cal G})_0$
    to $PB_k({\cal G})_0$, and by suitably modifying Lemmas \ref{pb} and \ref{pb1}.}\end{rem}

We now prove Proposition \ref{mt}, which gives the counter
examples to the Fadell-Neuwirth
fibration theorem, in the context of Lie groupoids. That is, we show that
Theorem \ref{fibration} is not true in general. 

\begin{proof}[Proof of Proposition \ref{mt}] 
  Recall that $S=M/H$, where $H$ is a finite group acting effectively
  on a connected manifold $M$ of dimension $\geq 2$.

  By definition the object space of the Lie groupoid $PB_n^*({\cal G}(M, H))$ is
  as follows. Here $*=a$ or $b$.
  $$PB_n({\cal G}(M,H))_0=\{(x_1,x_2,\ldots, x_n)\in M^n\ |\ Hx_i\neq Hx_j,\ \text{for}\ i\neq j\}.$$
\noindent
  Next, consider the projection $f:PB_n({\cal G}(M,H))_0\to PB_{n-1}({\cal G}(M,H))_0$ to
  the first $n-1$ coordinates. Note that $f=F^a_0=F^b_0$. We will show that $f$
  is not a fibration.
  
  Since $S$ has at least one singular point, there is a point $s\in M$, so that
  the isotropy group $H_s=\{h\in H\ |\ hs=s\}$ is nontrivial.  
  There is a neighbourhood $U_s\subset M$ of $s$ preserved by $H_s$ and 
  $hU_s\cap U_s=\emptyset$ for all $h\in H - H_s$. Such a
  neighbourhood exists, see the proof of [\cite{Thu}, Proposition 5.2.6].
  Since regular points are dense in $S$, there is a point $s'\in U_s$
  which has trivial isotropy group. That is, $s'$ corresponds
  to a regular point and $s$ corresponds to a singular point on $S$.

   Choose a point $p=(s, x_2,\ldots,
  x_{n-1})\in PB_{n-1}({\cal G}(M,H))_0$, such that
  $Hx_i\neq Hs'$ for all $i=2,3,\ldots, n-1$. Let $p'=(s', x_2,\ldots, x_{n-1})$.
  Note that $|Hs| < |Hs'|$. 
  Then it follows that $f^{-1}(p)$ and $f^{-1}(p')$ are obtained from
  $M$, by removing different number of points. Hence, they are not
  homotopy equivalent. Therefore, $f$ is not a fibration. 

  In fact, we have proven that $f$ is not even a
quasifibration. Since a quasifibration with a path connected base
space demands that any two fibers
should have isomorphic homotopy groups.

This proves Proposition \ref{mt}. 
\end{proof}

Next, we give the proof of the short exact sequence of orbifold
fundamental groups of the
configuration orbifold of the orbifold ${\Bbb C}(k,m;q)$. This proof is the
crucial application of the explicit set of generators of the pure orbifold braid
group, we constructed in Lemma \ref{orbi-gen}.

\begin{proof}[Proof of Theorem \ref{esg}]
By Remark \ref{es}, we have to prove the exactness of the following
sequence.

\centerline{
\xymatrix{1\ar[r]&K\ar[r]&\pi_1^{orb}(PB_n(S))\ar[r]^{\Delta}&\pi_1^{orb}(PB_{n-1}(S))\ar[r]&1.}}
\noindent
Where, $S={\Bbb C}(k,m;q)$ and $F=S-\{(n-1)-\text{regular points}\}$ and $K$
is isomorphic to $\pi_1^{orb}(F)$.

First, let us recall that, when $S={\Bbb C}$ then the above exact
sequence follows from Corollary \ref{FNF}.

The map $\pi_1(PB_n({\Bbb C}))\to \pi_1(PB_{n-1}({\Bbb C}))$ is  
obtained by removing the last string in a braid, representing an
element 
of $\pi_1(PB_n({\Bbb C}))$. This is the main idea to understand the
homomorphism $\Delta$ in terms
of braid pictures, and to prove the above exact sequence.

Now, following the braid presentation
of elements of $\pi_1^{orb}(B_n({\Bbb C}(k,m; q)))$
as was done in Section \ref{obg}, we know 
that a typical element of $\pi_1^{orb}(PB_n({\Bbb C}(k,m; q)))$ 
looks like $A$ as in Figure 4.

So, we send $A$ 
to the braid $\Delta (A)\in \pi_1^{orb}(PB_{n-1}({\Bbb C}(k,m; q)))$ after
removing the last string from $n$ to $n$, as is shown in Figure 15.
 
\medskip

\centerline{\includegraphics[height=4cm,width=8cm,keepaspectratio]{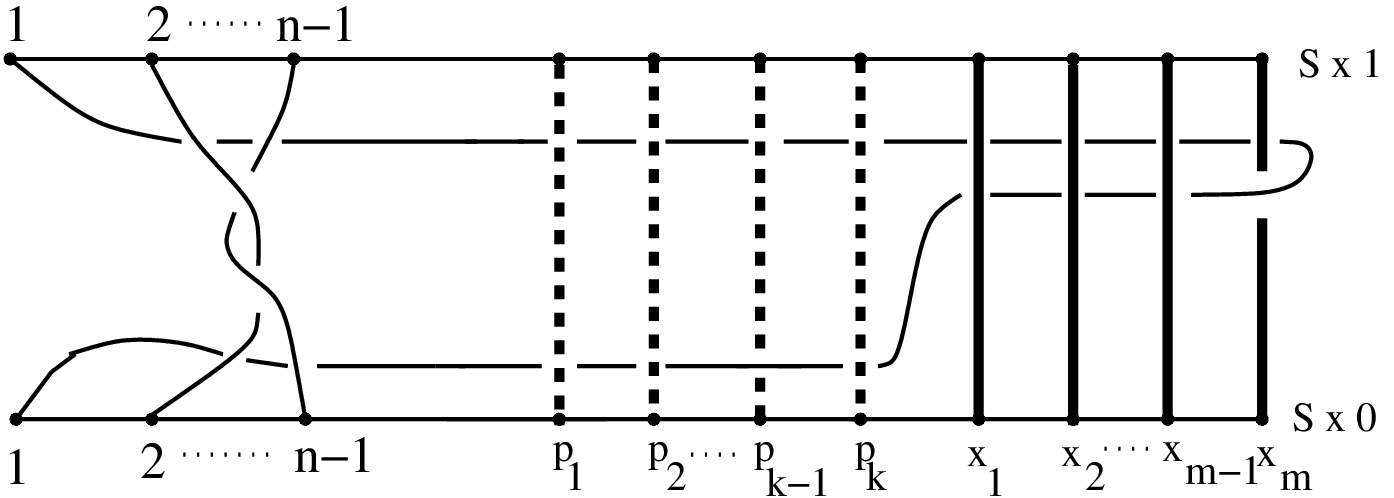}}

\centerline{Figure 15: The braid $\Delta (A)$ in $\pi_1^{orb}(PB_{n-1}({\Bbb C}(k,m; q)))$.}

\medskip

Clearly, $\Delta$ is a surjective homomorphism, as the law of composition is
juxtaposition of braids. This also follows from the fact that, by adding the extra
relations $X_{nr}, r=1,2,\ldots ,m;\ P_{ns}, s=1,2,\ldots , k$ and $B_{in},
i=1,2,\ldots , n-1$ in the presentation of $\pi_1^{orb}(PB_n({\Bbb C}(k,m; q)))$, we get the
presentation of $\pi_1^{orb}(PB_{n-1}({\Bbb C}(k,m; q)))$. Since, any element of
$\pi_1^{orb}(PB_n({\Bbb C}(k,m; q)))$ is a juxtaposition of the generators as in
Lemma \ref{orbi-gen}, therefore, removing the string from $n$ to $n$ is equivalent to attaching the
extra relations as above.

Recall that,
we denoted the above set of braids by $\cal N$ in Proposition \ref{gen}. A splitting
map is defined by sending $\Delta (A)$ to an element $\tilde A$, with the string from $n$ to $n$
not entangling with any other strings, the dotted or the thick
lines and going over all of them, as shown in Figure 16. It is again
easy to see that this splitting is a homomorphism. 

\medskip

\centerline{\includegraphics[height=4cm,width=8cm,keepaspectratio]{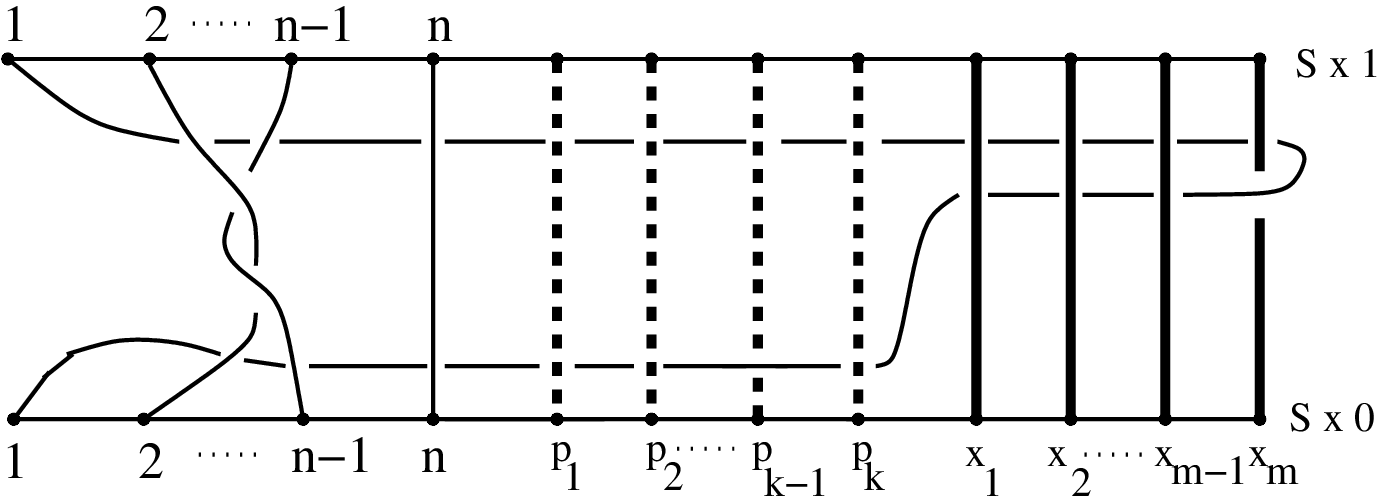}}

\centerline{Figure 16: The braid $\tilde A$ in $\pi_1^{orb}(PB_n({\Bbb C}(k,m; q)))$.}

\medskip

Clearly, the braids in the set $\cal N$ belong
  to the kernel of $\Delta$.

  Let $\Delta (A)$ be the trivial element in
$\pi_1^{orb}(PB_{n-1}({\Bbb C}(k,m; q)))$. 

We want to prove that, $A$ is equivalent to a
juxtaposition of finitely many members of
$\cal N$ or its inverses. Equivalently, we
have to show that the braids in $\cal N$ generate a
subgroup, which is normal in $\pi_1^{orb}(PB_n({\Bbb C}(k,m; q)))$.
Hence, we apply Proposition \ref{gen}, to see that the kernel of $\Delta$ is generated
by $\cal N$.

Now, recall that $X_{nr}$ is of order $q_{r}$, and the remaining two sets of
generators are of infinite order.

We will now show that the elements in $\cal N$
freely generate the kernel of $\Delta$, except the relations, $X_{nr}^{q_r}, r=1,2,\ldots ,m$.
That is, the  kernel of $\Delta$ is isomorphic to
$$C_{q_1}*C_{q_2}*\cdots *C_{q_m}*C*C*\cdots *C\simeq\pi_1^{orb}(F).$$
Here, $C_{q_i}$ is the cyclic group of order $q_i$, $C$ is infinite cyclic and there are $k+n-1$
number of factors of $C$ in the above display.

First, recall that by Corollary \ref{FNF},
the kernel of $f_*:PB_n({\Bbb C})\to PB_{n-1}({\Bbb C})$ is
a free group on $n-1$
generators, and it is generated by $B^{(n)}_{in}$ for
$i=1,2,\ldots,n-1$.

Now, we replace the dotted and the thick lines in the generators 
$X_{nr}, P_{ns}$ and $B_{in}$ of the kernel of $\Delta$ by straight strings,
number the strings and denote them by 

 $$\tilde X_{nr}, r=1,2,\ldots ,m;\ \tilde P_{ns}, s=1,2,\ldots , k;\ \tilde B_{in},
 i=1,2,\ldots , n-1.$$

 These braids look as follows.
 
 \medskip

\centerline{\includegraphics[height=3cm,width=6cm,keepaspectratio]{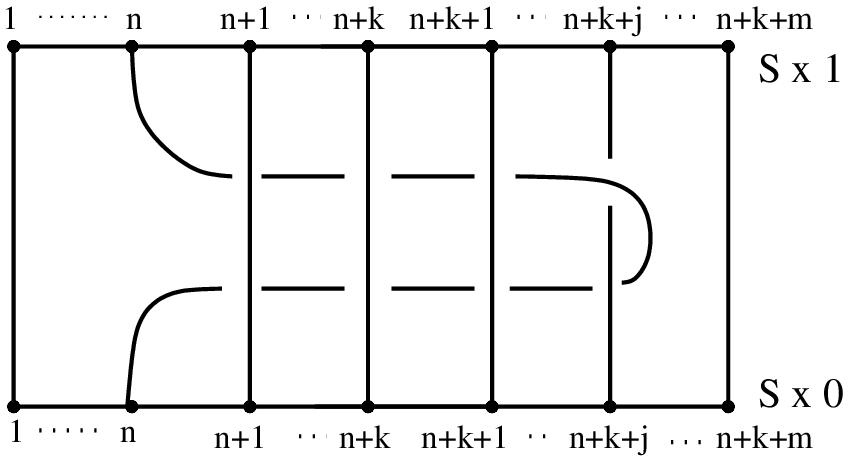}\
  \ \includegraphics[height=3cm,width=6cm,keepaspectratio]{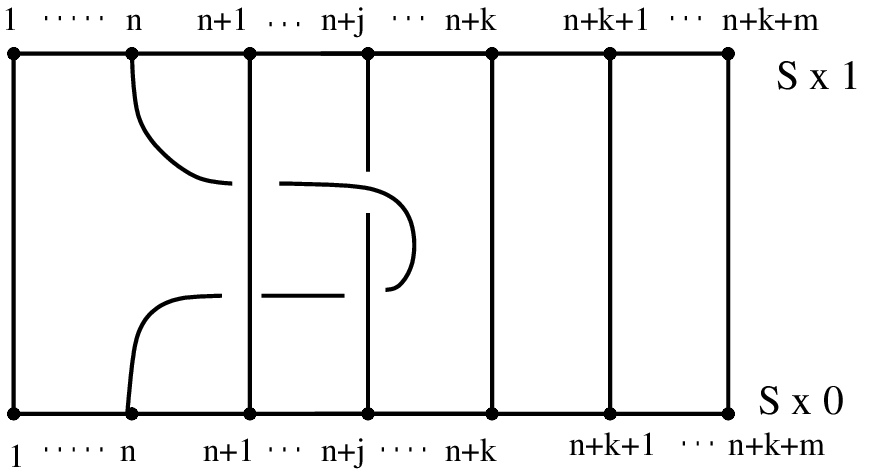}}

\centerline{Figure 17: $\tilde X_{nr}$.\ \ \ \ \ \ \ \ \ \ \ \ \ \ \ \ \ \ \ \ \ \ \ Figure 18: $\tilde P_{ns}$.}

\medskip

\centerline{\includegraphics[height=3cm,width=6cm,keepaspectratio]{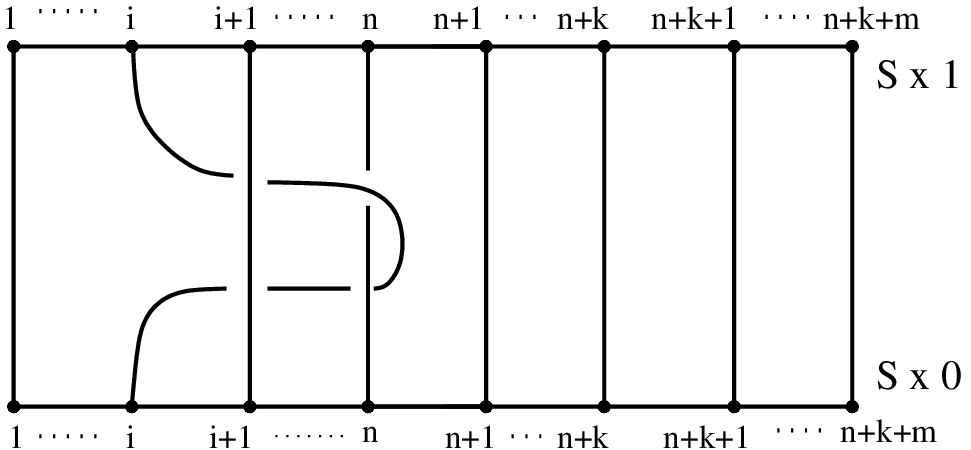}}

\centerline{Figure 19: $\tilde B_{in}$.}

\medskip

 Then, $\tilde {\cal N}=\{\tilde X_{nr}, \tilde P_{ns}, \tilde B_{in}\}$ 
 generates a free group, since it is easily identified (by permuting
 the coordinates using the transposition $(n, n+k+m)$)
 with the kernel of   
$$\pi_1(PB_{m+k+n}({\Bbb C}))\to \pi_1(PB_{m+k+n-1}({\Bbb C})),$$ 
generated by
$$\{B^{(n+k+m)}_{in}, B^{(n+k+m)}_{nj},\ \text{for}\ 
j=n+1,\ldots, n+k+m; \ \text{and}\ i<n\},$$
and we know that the above kernel is a free
group by Corollary \ref{FNF}.

Next, we put back the dotted and the thick lines (using the
 map ${\cal O}_{n+k+m}$ in Remark \ref{G}), then, it is clear
from Section \ref{obg}, 
that we only have to put the relations $X_{nr}^{q_{r}},
r=1,2,\ldots ,m$ to get a 
complete set of relations in the presentation of the kernel of
$\Delta$.   
Since, there is no relations in a presentation of the subgroup generated by $\tilde {\cal N}$
in $\pi_1(PB_{m+k+n}({\Bbb C}))$, we get that after removing the $\tilde{}$, there is no 
relations in a presentation of the subgroup generated
by $\cal N$ in $\pi_1^{orb}(PB_n({\Bbb C}(k,m; q)))$, except the above
finite order relations. This is because ${\cal O}_{n+k+m}$ is a homomorphism.
Therefore, the kernel of $\Delta$ is isomorphic
to $\pi^{orb}_1(F)$.

This completes the proof of the theorem.
  \end{proof}

  Now, we are in a position to give the proof of Theorem \ref{mtpf}, 
  giving virtual poly-$\cal F$ structure on some affine and finite complex Artin groups.

  To prove the theorem, furthermore, we need the following two results.

\begin{thm}\label{AO} (\cite{MS}) All affine type Artin groups are torsion free.\end{thm}

\begin{thm}(\cite{All}) \label{All} Let $\cal A$ be an Artin group, and $\cal O$ be an orbifold as 
described in the following table. Then, $\cal A$ can be embedded as a normal 
subgroup in $\pi_1^{orb}(B_n({\cal O}))$. The third column gives the quotient
group $\pi_1^{orb}(B_n({\cal O}))/{\cal A}$.

\medskip
\centerline{
\begin{tabular}{|l|l|l|l|}
\hline
{\bf Artin group of type} & {\bf Orbifold $\cal O$}&{\bf Quotient group}&$n$\\
\hline \hline
&&&\\
$B_n$& ${\Bbb C}(1,0)$ &$<1>$& $n>1$\\
\hline
&&&\\
$\tilde A_{n-1}$& ${\Bbb C}(1,0)$ &${\Bbb Z}$& $n>2$\\
\hline
&&&\\
$\tilde B_n$&${\Bbb C}(1,1;(2))$  &${\Bbb Z}/2$& $n>2$\\
\hline
&&&\\
$\tilde C_n$&${\Bbb C}(2,0)$ &$<1>$&$n>1$\\
  \hline
  &&&\\
  $\tilde D_n$&${\Bbb C}(0,2;(2,2))$&${\Bbb Z}_2\times{\Bbb Z}_2$&$n>2$\\
  \hline
\end{tabular}}

\medskip
\centerline{\rm{Table: Embedding Artin groups in orbifold braid groups}}
\end{thm}
  
We also need the following.

\begin{prop}\label{fi} If a torsion free, finitely presented poly-$\cal {VF}$ group
  has a normal series with finitely presented subgroups, then the group is
  virtually poly-$\cal F$.\end{prop}

\begin{proof} Let $H$ be a poly-$\cal {VF}$ group of length $n$ satisfying the 
  hypothesis of the statement.
  The proof is by induction on $n$. If $n=1$, then $H$ is virtually free, and hence free, since it is
  torsion free (\cite{S}). Therefore, assume that the lemma is true for all 
  poly-$\cal {VF}$ groups of length $\leq n-1$ satisfying the hypothesis. Consider a
  finitely presented normal series for $H$ giving the poly-$\cal {VF}$ structure.
  Then, $H_{n-1}$ is finitely presented, torsion free and has a poly-$\cal {VF}$  
  structure of length $n-1$. Hence, by the induction hypothesis, there is a finite index
  subgroup $K\leq H_{n-1}$ and $K$ is poly-$\cal F$. Since $H_{n-1}$ is finitely
  presented, we can find a finite index subgroup $K'$ of $K$ which is 
  also a characteristic subgroup of $H_{n-1}$. Hence, $K'$ is a poly-$\cal F$  
  group, and is a normal subgroup of $H=H_n$ with quotient virtually
  free. Let $q:H\to H/K'$ be the quotient map. Choose a free subgroup $L$ of $H/K'$ of finite
  index, then $q^{-1}(L)$ is a finite index poly-$\cal F$ subgroup of $H$.

  This proves the Proposition.\end{proof}

Now, we are ready to prove Theorem \ref{mtpf}.

\begin{proof}[Proof of Theorem \ref{mtpf}] 
  From the Table in Theorem \ref{All}, we see
  that the Artin group of type
  $\tilde A_n$ is a subgroup of the finite type Artin group of type $B_{n+1}$. Hence,
  by \cite{Br}, the Artin group of type $\tilde A_n$ is virtually poly-$\cal F$.
  
  Now, let $\cal A$ be an Artin group of type $\tilde B_n$, $\tilde C_n$ or $\tilde D_n$.
  Then, by Theorem \ref{All}, $\cal A$ can be embedded as a 
  normal subgroup in $\pi_1^{orb}(B_n({\Bbb C}(k,m;q)))$ of finite index, for some suitable $k,m,q$ and $n$.
  We know by Corollary \ref{vfg} that, $\pi_1^{orb}(PB_n({\Bbb C}(k,m;q)))$ is poly-$\cal {VF}$ by
  a normal series consisting of finitely presented subgroups. Therefore,
  the same is true for $\pi_1^{orb}(B_n({\Bbb C}(k,m;q)))$, since the
  pure orbifold braid group is embedded as a finite index normal subgroup in the
  orbifold braid group. Next, since 
  $\cal A$ is a subgroup of finite index in
  $\pi_1^{orb}(B_n({\Bbb C}(k,m;q)))$, it follows that 
  $\cal A$ is also poly-$\cal F$, by a normal series consisting of finitely presented
  subgroups. But by Theorem \ref{AO}, $\cal A$ is
  also torsion free. Hence, by Proposition \ref{fi} $\cal A$ is virtually poly-$\cal F$.

  The $G(de,e,r)$ ($d,r\geq 2$) type case is easily deduced from the fact that, 
  this Artin group can be embedded as a subgroup in the finite type Artin group of type
  $B_r$. See [\cite{CLL}, Proposition 4.1].
  
Therefore, we have completed the proof of Theorem \ref{mtpf}.\end{proof}

We now give the proof of Theorem \ref{FJC}, that is, the proof of
the isomorphism conjecture for the orbifold braid group of the orbifold
${\Bbb C}(k,m;q)$. 

  \begin{proof}[Proof of Theorem \ref{FJC}] We start with the following
    definition of a class of groups, which contains the class of groups $\cal C$ 
    defined in [\cite{Rou}, Definition 3.1].

    \begin{defn}\label{C}{\rm Let ${\cal D}$ denote the smallest class of groups 
satisfying the following conditions.

1. The fundamental group of any connected manifold of dimension $\leq 3$ belongs 
to ${\cal D}$.

2. If $H$ is a subgroup of a group $G$, then $G\in {\cal D}$ 
implies $H\in {\cal D}$. This reverse implication is also true if 
$H$ is of finite index in $G$.

3. If $G_1, G_2\in {\cal D}$ then $G_1\times G_2\in {\cal D}$.

4. If $\{G_i\}_{i\in I}$ is a directed system of groups and $G_i\in {\cal D}$ for each
$i\in I$, then the $\lim_{i\in I}G_i\in {\cal D}$. 

5. Let

\centerline{
\xymatrix{1\ar[r]&K \ar[r]&G\ar[r]^p&H\ar[r]&1}}

\medskip
\noindent
be a short exact sequence of groups. If $K$, $H$ and 
$p^{-1}(C)$, for any infinite cyclic subgroup $C$ of $H$, belong to 
$\cal D$ then $G$ also belongs to $\cal D$.

6. If a group $G$ has a normal subgroup $H$, so that $H$ is free and
$G/H$ is infinite cyclic, then $G$ belongs to $\cal D$.}\end{defn}

Note that, [\cite{Rou}, Definition 3.1] did not have the condition $6$.
In [\cite{Rou}, Theorem 3.3] we noted that the Farrell-Jones isomorphism conjecture
with coefficients, and finite wreath product is true for any group belonging
to the class $\cal C$. Two recent papers (\cite{BFW} and \cite{Wu}) help
us to conclude that the conjecture is
true for any member of $\cal D$. In $6$, when $H$ is finitely
generated, then the conjecture is proved in \cite{BFW}, and it is generalized for arbitrary free
group case in \cite{Wu}.

We will now prove that, the orbifold braid groups of the
orbifold ${\Bbb C}(k,m;q)$ belong to $\cal D$. 

Recall the following exact sequence from Remark \ref{es}. Here, $S$ is the
    orbifold ${\Bbb C}(k,m;q)$.
    
\centerline{
\xymatrix{1\ar[r]&\pi_1^{orb}(F)\ar[r]&\pi_1^{orb}(PB_n(S))\ar[r]&\pi_1^{orb}(PB_{n-1}(S))\ar[r]&1.}}
\noindent
Where $F=S-\{(n-1)-\text{regular points}\}$.

Note that, $\pi_1^{orb}(B_n({\Bbb C}(k,m; q)))$ contains
$\pi_1^{orb}(PB_n({\Bbb C}(k,m; q)))$ as a subgroup
of finite index, and hence, by $2$, it is enough to prove
that $\pi_1^{orb}(PB_n({\Bbb C}(k,m; q)))\in {\cal D}$.

The proof is by induction on $n$. Note that, for $n=1$,
$\pi_1^{orb}(PB_n({\Bbb C}(k,m; q)))\simeq \pi_1^{orb}({\Bbb C}(k,m; q))$, which
is virtually finitely generated free (see Corollary \ref{vfg}), and
hence by $1$ and $2$,  
$\pi_1^{orb}({\Bbb C}(k,m; q))\in {\cal D}$.
By the same argument $\pi_1^{orb}(F)\in {\cal D}$. We would like
to use $5$ now. Therefore, we assume
$\pi_1^{orb}(PB_{n-1}({\Bbb C}(k,m; q)))\in {\cal D}$ and let $C$ be an infinite
cyclic subgroup of $\pi_1^{orb}(PB_{n-1}({\Bbb C}(k,m; q)))$. The inverse image of $C$ under the 
surjective homomorphism in the above display is an extension of $\pi_1^{orb}(F)$ by $C$, that is,
isomorphic to the semi-direct product $\pi_1^{orb}(F)\rtimes C$. Since, $\pi_1^{orb}(F)$ has a
finitely generated free subgroup of finite index, it is easy to deduce that it has a
finitely generated free characteristic subgroup $K$ (say) of finite index. Therefore,
the action of $C$ on $\pi_1^{orb}(F)$ preserves $K$ and hence, $K\rtimes C$ is a subgroup
of $\pi_1^{orb}(F)\rtimes C$ of finite index. Therefore, $\pi_1^{orb}(F)\rtimes C\in{\cal D}$
using $6$ and $2$. Hence, $\pi_1^{orb}(PB_n({\Bbb C}(k,m; q)))\in {\cal D}$ by $5$.

Now, we see using the Table in Theorem \ref{All}, that the Artin group of type $\tilde D_n$ is a subgroup
of the orbifold braid group of ${\Bbb C}(0,2;(2,2))$. Therefore, the Artin group of type
$\tilde D_n$ belongs to $\cal D$, using $2$.

This completes the proof of the theorem.
\end{proof}

\newpage
\bibliographystyle{plain}
\ifx\undefined\bysame
\newcommand{\bysame}{\leavevmode\hbox to3em{\hrulefill}\,}
\fi

\end{document}